\documentclass[11pt,twoside, a4paper]{amsart}

\usepackage[english]{babel} 
\usepackage[T1]{fontenc} 
\usepackage{amsmath}
\usepackage{amssymb} 
\usepackage{float}
\usepackage{amsfonts}
\usepackage{mathtools}
\usepackage[utf8]{inputenc}
\usepackage[mathcal]{eucal} 
\usepackage{enumerate}
\usepackage{amsthm} 
\usepackage{hyperref}
\usepackage[totalwidth=14cm,totalheight=20cm, hmarginratio=1:1]{geometry}

\newtheorem{defi}{Definition}[section] 
\newtheorem{teo}[defi]{Theorem}
\newtheorem{cor}[defi]{Corollary} 
\newtheorem{lemma}[defi]{Lemma}

\newtheorem{prop}[defi]{Proposition}
\newtheorem*{notation}{Notation}

\theoremstyle{definition}
\newtheorem{oss}{Remark}[section]

\newcommand{\Pp}{\mathbb{P}}
\newcommand{\R}{\mathbb{R}}

\newcommand{\h}{\mathbb{H}}
\newcommand{\N}{\mathbb{N}}

\newcommand{\dPSL}{\mathbb{P}SL(2,\mathbb{R})\times \mathbb{P}SL(2,\mathbb{R})}

\newcommand{\PSL}{\mathbb{P}SL}
\newcommand{\T}{\mathcal{T}}

\newcommand{\Ima}{\mathrm{Im}}

\newcommand{\Imm}{\mathcal{I}m}

\newcommand{\trace}{\mathrm{tr}}

\newcommand{\cro}{\mathrm{cr}}

\newcommand{\B}{\mathcal{B}}
\newcommand{\D}{\mathcal{D}}
\newcommand{\C}{\mathcal{C}}
\newcommand{\PP}{\mathcal{P}}
\newcommand{\GH}{\mathcal{GH}}
\newcommand{\Ein}{\mathrm{Ein}}
\newcommand{\arccosh}{\mathrm{arccosh}}
\newcommand{\hol}{\mathrm{hol}}
\newcommand{\dev}{\mathrm{dev}}
\newcommand{\Hom}{\mathrm{Hom}}
\newcommand{\tw}{\mathrm{tw}}
\newcommand{\Ree}{\mathcal{R}e}

\DeclareMathAlphabet{\mathpzc}{OT1}{pzc}{m}{it}
\DeclareMathOperator{\arctanh}{arctanh}

\title[Fenchel-Nielsen coordinates in AdS geometry]{Fenchel-Nielsen coordinates on the augmented moduli space of anti-de Sitter structures}
\author{Andrea Tamburelli}
\date{\today}
\thanks{}

\begin{document}

\begin{abstract} 
In this paper we combine our recent work on regular globally hyperbolic maximal anti-de Sitter structures with the classical theory of globally hyperbolic maximal Cauchy-compact anti-de Sitter manifolds in order to define an augmented moduli space. Moreover, we introduce a coordinate system in this space that resembles the complex Fenchel-Nielsen coordinates for hyperbolic quasi-Fuchsian manifolds.
\end{abstract}

\maketitle
\setcounter{tocdepth}{1}
\tableofcontents

\section*{Introduction} 
In the last few years people have been interested in the study of surface group representations into higher rank Lie groups aiming at understanding to which extent the well-known Teichm\"uller theory for $\PSL(2,\R)$ can be generalised (\cite{Wienhard_ICM}). One example of higher Teichm\"uller space is the deformation space of globally hyperbolic maximal Cauchy-compact (GHMC) anti-de Sitter structures $\mathcal{GH}(S)$ on $S\times \R$, where $S$ is a closed, orientable surface of genus at least $2$. This was introduced by the pioneering work of Mess (\cite{Mess}) and can be described as the space of maximal representations of $\pi_{1}(S)$ into $\dPSL$, up to conjugation. Recently, we extended this theory to surfaces with punctures (\cite{Tambu_poly}, \cite{Tambu_regularAdS}, \cite{wildAdS}) and defined regular globally hyperbolic maximal anti-de Sitter manifolds: these are diffeomorphic to $S\times \R$, where now $S$ is a surface with punctures and negative Euler characteristic, and we allow the projection of the holonomy around each puncture onto any of the two $\PSL(2,\R)$ factors to be hyperbolic or parabolic. However, in this case, the holonomy does not suffice to describe the structure uniquely and new extra data about the behaviour of the developing map must be included. Moreover, in a recent paper (\cite{Tambu_degeneration}) we showed that regular globally hyperbolic maximal anti-de Sitter structures naturally appear as geometric limits along pinching sequences of GHMC anti-de Sitter manifolds.\\

The aim of this paper is to give a unified picture of the two theories by defining an \emph{augmented} deformation space of globally hyperbolic maximal anti-de Sitter structures $\mathcal{GH}(S)^{aug}$ and by providing a coordinate system, resembling the complex Fenchel-Nielsen coordinates for hyperbolic quasi-Fuchsian manifolds (\cite{MR1266284}, \cite{MR1196924}), that varies continuously under pinching of a multi-curve, up to the action of a suitable subgroup of the mapping class group. To this aim, we will use an upper-half space model of anti-de Sitter space introduced by Danciger (\cite{Jeff_thesis}), which will make the analogy with the hyperbolic setting very explicit. We will describe the model in details in Section \ref{sec:model}. For now, it suffices to know that it is constructed using the $\R$-algebra $\B$ generated by an element $\tau$ with the property that $\tau^{2}=1$ and identifies the group of orientation and time-orientation isometries of anti-de Sitter space as $\PSL(2,\B)$. For every simple closed curve in a pants decomposition of $S$, we will define a $\B$-length and a $\B$-twist-bend parameter and we will prove the following:\\
\\
{\bf{Theorem A.}} {\emph{Let $S$ be a closed, orientable surface of genus $g\geq 2$. Fix a pants decomposition $\PP$ of $S$. There is a homeomorphism
\[
    \Phi_{\PP}: \mathcal{GH}(S) \rightarrow (\mathcal{C}^{+}\times \B)^{3g-3} \ ,
\]
where $\C^{+}$ denotes the positive cone of space-like points in $\B$ with positive real part, obtained by associating to a GHMC anti-de Sitter structure the $\B$-lengths and the $\B$-twist-bend parameters of the curves in $\PP$.}}\\
\\
The above result, with slight modifications, also applies to regular GHM anti-de Sitter structures in case of surfaces with punctures.\\

The main result of the paper explains how to extend these coordinates to the augmented deformation space when we allow the structure to degenerate along a multi-curve. We will see that the augmented deformation space is naturally stratified $\mathcal{GH}(S)^{aug}=\bigcup_{\D}\mathcal{GH}(S,\D)^{reg}$, where each stratum is determined by multi-curve $\D$ on $S$, and we will show the following:\\
\\
{\bf{Theorem B.}} {\emph{Let $\mu \in \mathcal{GH}(S)^{aug}$ and let $\D$ be a multi-curve so that $\mu \in \mathcal{GH}(S,\D)^{reg}$. Fix a pants decomposition $\PP \supset \D$ of $S$. Let $G_{\D}$ be the subgroup of the mapping class group generated by Dehn-twists about the simple closed curves in $\D$. Then
\[
    \GH(S)^{aug,\D}:=\bigcup_{\D'\subset \D}\GH(S,\D')^{reg}
\]
is an open subset of $\GH(S)^{aug}$ containing $\mu$ and invariant under $G_{\D}$. Moreover, there is a homeomorphism
\[
    \Phi_{\PP, \D}: \GH(S)^{aug, \D}/G_{\D}\rightarrow (\C^{+}\times \B)^{3g-3-m}\times \R^{4m}
\]
where $m$ is the number of curves in $\D$.}}\\
\\
The coordinates are an explicit extension of the $\B$-lenghts and the $\B$-twist-bend parameters in Theorem A. This theorem should be interpreted as a generalisation of the standard result in Teichm\"uller theory about the behaviour of Fenchel-Nielsen coordinates at the boundary of the augmented Teichm\"uller space of $S$ and as the anti-de Sitter analogue of a recent theorem by Loftin and Zhang (\cite{Loftin_Zhang}) in the context of convex real projective structures.

\subsection*{Outline of the paper} In Section \ref{sec:model} we recall the main features of the upper-half space model of anti-de Sitter space. In particular, we describe the analogues of shearing and bending in anti-de Sitter geometry which will allow us to define Fenchel-Nielsen coordinates in Section \ref{sec:coord}. In Section \ref{sec:GHMC} we study the augmented deformation space of globally hyperbolic maximal anti-de Sitter structures and its topology. Section \ref{sec:mainthm} deals with the proof of Theorem B.

\section{Upper-half space model of anti-de Sitter space}\label{sec:model}
In his thesis (\cite{Jeff_thesis}), Danciger introduced an upper-half space model for anti-de Sitter space using generalised Clifford numbers. As it may not be familiar to the reader, we recall here the main features and properties of this model that we are going to use in the sequel.

\subsection{The algebra $\B$} Let $\B$ be the $\R$-algebra generated by an element $\tau$ with $\tau^{2}=1$. In particular, it is a two-dimensional vector space and we will often denote $\B=\R \oplus \tau\R$. Borrowing terminology from the complex numbers, we talk about real and imaginary part of an element of $\B$ and we define a conjugation 
\[
    \overline{a+\tau b}=a-\tau b \ .
\]
This induces a norm on $\B$ by taking
\[
    |a+\tau b|=(a+\tau b)\overline{(a+\tau b)}=a^{2}-b^{2} \ .
\]
Notice that elements of $\B$ may have negative norm. In fact, the bi-linear extension of this norm defines a Minkowski inner product on $\R^{2}$ with orthonormal basis $\{1,\tau\}$. We call space-like the elements of $\B$ of positive norm: these form a subgroup of $\B$ under multiplication. It is also convenient to work with the basis of idempotents
\[
    e^{+}=\frac{1+\tau}{2} \ \ \ \text{and} \ \ \  e^{-}=\frac{1-\tau}{2}
\]
as the map 
\[
    ae^{+}+be^{-} \mapsto (a,b)
\]
gives an isomorphism of $\R$-algebras between $\B$ and $\R\oplus \R$, where in the latter operations are carried out component by component. Let $M(2,\B)$ denote the algebra of two-by-two matrices with coefficients in $\B$. Decomposing an element of $M(2, \B)$ into $Ae^{+}+Be^{-}$, where $A$ and $B$ have real coefficients, gives an isomorphism between $M(2, \B)$ and $M(2,\R)\times M(2,\R)$. Moreover, because 
\[
    \det(Ae^{+}+Be^{-})=\det(A)e^{+}+\det(B)e^{-}
\]
this descends to an isomorphism $\PSL(2, \B) \cong \dPSL$. \\

The compactification of $\B$ is defined by
\[
    \Pp\B^{1}=\{(x,y) \in \B^{2} \ | \ \text{if $\alpha x=\alpha y=0$ then $\alpha=0$} \ \}/ \sim
\]
with $(x,y)\sim (\lambda x, \lambda y)$ for every $\lambda \in \B^{*}$. As common, we embed $\B$ into $\Pp\B^{1}$ by identifying
\[
    \B=\{ [x,1]\in \Pp\B^{1} \ | \ x \in \B \}\ .
\]
Notice that not only one point has been added at infinity (which happens when working with coefficients in a field). Again, the decomposition 
\[
    [ae^{+}+be^{-}, ce^{+}+de^{-}]=[a,c]e^{+}+[b,d]e^{-}
\]
provides a homeomorphism between $\Pp\B^{1}$ and $\R\Pp^{1} \times \R\Pp^{1}$, which is compatible with the action of $\PSL(2, \B)$ on $\Pp\B^{1}$ and the diagonal action of $\dPSL$ on $\R\Pp^{1}\times \R\Pp^{1}$. It turns out that $\Pp\B^{1}$ is homeomorphic to the Lorentzian compactification of Minkowski space, thus it is a copy of the $2$-dimensional Einstein Universe $\Ein^{1,1}$ (\cite[pag. 102]{Jeff_thesis}). 

\subsection{Upper-half space model} Starting from the algebra $\B$, we add an 
element $j$ such that $j^{2}=-1$ and $j\tau=-\tau j$. This defines a $4$-dimensional algebra $\mathcal{A}$ over the reals that is isomorphic to the algebra of real two-by-two matrices. We can extend linearly the conjugation on $\B$ to elements of $\mathcal{A}$ by defining
\[
    \overline{j}=-j \ \ \ \text{and} \ \ \ \overline{zw}=\bar{z}\bar{w}\ ,
\]
and consider the square-norm given by $|z|^{2}=z\overline{z} \in \R$. Let $V=\mathrm{Span}\{1, \tau, j\}$. The bi-linear extension of the square-norm restricted to $V$ induces an inner-product of signature $(2,1)$ with orthogonal basis $\{1, j, \tau\}$. We compactify $V$ by introducing
\[
    \Pp V=\{ (x,y) \in \mathcal{A}^{2} \ | \ x\overline{y} \in V, \ \text{if $x\alpha =y\alpha=0$ then $\alpha=0$}\} / \sim
\]
with $(x,y)\sim (x\lambda, y\lambda )$ for every $\lambda \in \mathcal{A}^{*}$. As usual, we identify $V$ with the subset of $\Pp V$ consisting of points with coordinates $[x,1]$ with $x \in V$. In this way, $\Pp V$ coincides with the Lorentzian conformal compactification of $V$, obtained by adding all endpoints of lines of $V$ (\cite[pag. 142]{Jeff_thesis}). Thus $\Pp V$ is a copy of the $3$-dimensional Einstein Universe $\Ein^{2,1}$. Moreover, it is easy to check that $\PSL(2, \B)$ acts transitively and faithfully on $\Pp V$ by matrix multiplication.\\

The algebra homomorphism $\mathcal{I}: \mathcal{A} \rightarrow \mathcal{A}$ determined by $\mathcal{I}(1)=1, \mathcal{I}(\tau)=\tau$ and $\mathcal{I}(j)=-j$ defines an involution of $V$ that fixes $\B$ and extends to an involution of $\Pp V$ by setting $\mathcal{I}([x,y])=[\mathcal{I}(x), \mathcal{I}(y)]$. The quotient
\[
    X=\Pp V / \mathcal{I}
\]
is homeomorphic to a solid torus with boundary $\Pp\B^{1}$ and will be the upper-half space model of anti-de Sitter space, once we have introduced a Lorentzian metric on $X$ with constant sectional curvature $-1$.\\

We first define a Lorentzian metric on $V \setminus \B$, and then we extend it to $\Pp V \setminus \Pp\B^{1}$ using the transitive action of $\PSL(2, \B)$. Let $x_{1}, x_{2}, x_{3}$ be real coordinates on $V$ with respect to the basis $\{1, \tau, j\}$. In analogy to the upper-half space model in hyperbolic geometry, we consider the Lorentzian metric on $V\setminus \B$ given by 
\[
    ds^{2}_{[x,1]}=\frac{dx_{1}^{2}-dx_{2}^{2}+dx_{3}^{2}}{x_{3}^{2}} \ .
\]
We first check that $\PSL(2,\B)$ acts isometrically on $(V\setminus \B, ds^{2})$. Let 
\[
    A=\begin{pmatrix} a & b \\ c & d \end{pmatrix} \in \PSL(2, \B)
\]
and let $[x,1] \in V$. We assume that $A[x,1] \in V$, which means that $cx+d \in \mathcal{A}^{*}$. Under this assumption, $A[x,1]=[f(x),1]$, where $f(x)=(ax+b)(cx+d)^{-1}$. Differentiating the expression $f(x)(cx+d)=ax+b$ and applying it to a tangent vector $u\in V$ we obtain that (\cite[A.2. pag 145]{Jeff_thesis})
\[
    df_{x}(u)=(cx+d)^{-1}u(cx+d)^{-1} \ ,
\]
from which we deduce that the action of $A$ on $V$ by M\"obius transformation is conformal with respect to the Minkowski metric on $V$. Moreover, writing
\[
    f(x)=f_{1}(x)+f_{2}(x)\tau+f_{3}(x)j
\]
we see that $f_{3}(x)=\frac{x_{3}}{|cx+d|^{2}}$ and 
\[
    A^{*}ds^{2}=\frac{df_{1}^{2}-df_{2}^{2}+df_{3}^{2}}{f_{3}(x)^{2}}=\frac{dx_{1}^{2}-dx_{2}^{2}+dx_{3}^{2}}{f_{3}(x)^{2}|cx+d|^{4}}=\frac{dx_{1}^{2}-dx_{2}^{2}+dx_{3}^{2}}{x_{3}^{2}} \ ,
\]
hence the metric is preserved by the action of $\PSL(2,\B)$. We can then extend the metric $ds^{2}$ to the whole $\Pp V \setminus \Pp \B$ by requiring that $\PSL(2, \B)$ act isometrically, i.e. given a point $[x,y] \in \Pp V$ we set 
\[
    ds^{2}_{[x,y]}=A^{*}ds^{2}_{A[x,y]}
\]
where $A \in \PSL(2, \B)$ is any matrix such that $A[x,y] \in V$. Notice that this does not depend on the choice of $A$. Moreover, the classical computation for the curvature of the hyperbolic metric in the upper-half space model can be adapted to this Lorentzian setting and shows that $ds^{2}$ has constant sectional curvature $-1$. Since the involution $\mathcal{I}$ preserves the metric $ds^{2}$, this descends to a Lorentzian metric on $X$. In the next section we will describe geodesics in this model and show that $X$ is geodesically complete and time-like geodesics have length $\pi$, thus $X$ is isometric to anti-de Sitter space. Moreover, since the topological boundary $\Pp \B^{1}$ lies at infinite distance from any point in the interior of $X$ with respect to the metric $ds^{2}$, we will often call $\Pp\B^{1}$ the boundary at infinity of anti-de Sitter space. Notice that $\Pp\B^{1}$ inherits a conformal Lorentzian structure which coincides with that described in the previous section coming from seeing $\Pp\B^{1}$ as compactification of Minkowski space.

\subsection{Geodesics} Let $P=\mathrm{Span}\{1,j\}$: it is a totally geodesic plane isometric to $\h^{2}$. Therefore, the curve $\gamma(t)=e^{t}j$ is a unit-speed space-like geodesic in $X$. In order to find the general form of a space-like geodesic in this model, it is thus sufficient to translate $\gamma(t)$ by the isometric action of $ \PSL(2, \B)$. The result is analogous to the description of geodesics in the upper-half space model of hyperbolic space:

\begin{prop}\cite[Proposition 58]{Jeff_thesis}\label{prop:spacelikegeo} Let $p_{1}, p_{2} \in \B$ be such that the displacement $\Delta=(p_{1}-p_{2})/2$ is not light-like. Let $p=(p_{1}+p_{2})/2$ be the midpoint. Then the anti-de Sitter geodesic $\gamma$ connecting $p_{1}$ and $p_{2}$ is the conic in the affine plane $p+\mathrm{Span}\{\Delta, j\}$ defined by the equation 
\[
    |\gamma-p|^{2}=|\Delta|^{2} \ .
\]
If $\Delta$ is space-like, then $\gamma$ is a Euclidean ellipse, and if $\Delta$ is time-like, then $\gamma$ is a hyperbola. In either case, $\gamma$ is orthogonal to the boundary at infinity.
\end{prop}

Consider now the curve $\sigma(t)=[\sin(t)\tau+j, \cos(t)]$. It is straightforward to check that $\sigma(t)$ is a unit-speed, time-like curve of length $\pi$: since $\sigma(0)=[j,1]=[\mathcal{I}(j), \mathcal{I}(1)]=[-j,1]=\sigma(\pi)$, the curve closes up in $X$ after a period of $\pi$. Moreover, a standard but tedious computations using Christoffel symbols shows that $\sigma$ is indeed a geodesic. Therefore, any other time-like geodesic can be obtained by translating $\sigma$ by the isometric action of $\PSL(2, \B)$, and we get the following:

\begin{prop}\cite[Proposition 59]{Jeff_thesis} Let $p, \Delta \in \B$ with $|\Delta|^{2}<0$. Then the set of points $\gamma$ in the affine plane $p+ \mathrm{Span}\{\Delta, j\}$ satisfying
\[
    |\gamma-p|^{2}=-|\Delta|^{2}
\]
defines a time-like geodesic.
\end{prop}

Notice that the geodesic $\sigma$ corresponds in the proposition above to the choice of $p=0$ and $\Delta= \tau$. We conclude with the description of light-like geodesics:

\begin{prop}\cite[Proposition 60]{Jeff_thesis} The parameterised light-like geodesics that intersect the boundary at infinity at $p \in \B$ is given by 
\[
    \gamma(t)=[p+\frac{1}{t}v,1]
\]
where $v\in \mathrm{Span}\{1,\tau,j\}$ is a light-like vector.
\end{prop}

\subsection{Classification of isometries} The isomorphism $\PSL(2,\B)\cong \dPSL$ toghether with the standard classification of isometries of the hyperbolic plane into hyperbolic, parabolic and elliptic, provides a natural classification of isometries of anti-de Sitter space. Precisely, we say that an element in $\PSL(2, \B)$ is 
\begin{itemize}
    \item loxodromic, if it corresponds to a pair of hyperbolic isometries of $\h^{2}$;
    \item semi-loxodromic, if it is given by a hyperbolic and a parabolic element;
    \item parabolic, if it corresponds to a pair of parabolic isometries of $\h^{2}$.
\end{itemize}
We are now going to describe the main features of these classes of isometries, and introduce their $\B$-length.

\subsubsection{Loxodromic isometries} Up to conjugation, a loxodromic isometry is given by 
\[
    A=\begin{pmatrix} e^{\frac{\lambda}{2}} & 0 \\ 0 & e^{-\frac{\lambda}{2}} \end{pmatrix}e^{+}+\begin{pmatrix} e^{\frac{\mu}{2}} & 0 \\ 0 & e^{-\frac{\mu}{2}} \end{pmatrix}e^{-}
\]
for some positive real numbers $\lambda$ and $\mu$, which are the translation length of each diagonal matrix acting on the hyperbolic plane. It has four fixed points on $\Pp\B^{1}$ ($[1,0],[e^{+},e^{-}],[e^{-},e^{+}]$ and $[0,1]$) and leaves two space-like geodesic invariant. The invariant geodesic $\gamma(t)=e^{t}j\ \in V$ has endpoints $[1,0]$ and $[0,1]$, which are the repelling and attracting fixed points of $A$, thus we will consider it as the axis of the isometry $A$. Let us see more in details how the isometry $A$ acts on the geodesic $\gamma(t)$. Recalling that $e^{\pm}=\frac{1\pm \tau}{2}$, we have
\begin{align*}
    A\cdot e^{t}j&=[e^{\frac{\lambda}{2}}+e^{\frac{\mu}{2}}+\tau(e^{\frac{\lambda}{2}}-e^{\frac{\mu}{2}})]e^{j}j[e^{-\frac{\lambda}{2}}+e^{-\frac{\mu}{2}}+\tau(e^{\frac{-\lambda}{2}}-e^{-\frac{\mu}{2}})]^{-1}\\
    &=\frac{1}{4}[e^{\frac{\lambda}{2}}+e^{\frac{\mu}{2}}+\tau(e^{\frac{\lambda}{2}}-e^{\frac{\mu}{2}})]e^{j}j[e^{-\frac{\lambda}{2}}+e^{\frac{\mu}{2}}+\tau(e^{\frac{\lambda}{2}}-e^{\frac{\mu}{2}})]\\
    &=\frac{1}{4}[e^{\frac{\lambda}{2}}+e^{\frac{\mu}{2}}+\tau(e^{\frac{\lambda}{2}}-e^{\frac{\mu}{2}})][e^{\frac{\lambda}{2}}+e^{\frac{\mu}{2}}-\tau(e^{\frac{\lambda}{2}}-e^{\frac{\mu}{2}})]e^{t}j\\
    &=e^{\frac{\lambda+\mu}{2}}e^{t}j
\end{align*}
which shows that the isometry moves points on $\gamma(t)$ by a distance of $\frac{\lambda+\mu}{2}$. \\
By looking at the differential of $A$, we can also analyse the behaviour of a loxodromic isometry on the plane orthogonal to $\gamma$. An orthonormal frame for the tangent space at $\gamma(0)=j$ is given by $\{1, \tau, j\}$, where the first two vectors are orthogonal to $\dot{\gamma}(0)=j$. Its parallel transport at $A\gamma(0)=e^{\frac{\lambda+\mu}{2}}$ is $\{e^{\frac{\lambda+\mu}{2}}, e^{\frac{\lambda+\mu}{2}}\tau, e^{\frac{\lambda+\mu}{2}}j\}$. On the other hand, using the formula for the differential of an element of $\PSL(2, \B)$ provided in the previous section, we have
\[
    dA_{j}(1)=\frac{1}{2}(e^{\mu}+e^{\lambda}+\tau(e^{\lambda}-e^{\mu})) \ .
\]
Therefore, 
\[
    \langle dA_{j}(1), e^{\frac{\lambda+\mu}{2}} \rangle=
    \frac{1}{2}e^{\frac{\lambda+\mu}{2}}(e^{\lambda}+e^{\mu})e^{-\lambda-\mu}=\cosh\Big(\frac{\lambda-\mu}{2}\Big)
\]
which means that $A$ acts by a rotation of hyperbolic angle $\frac{\lambda-\mu}{2}$ on the orthogonal to $\dot{\gamma}$. We can encode both information in only one element of $\B$: we define the $\B$-length of the isomety $A$ as 
\[
    \ell_{\B}(A):=\frac{\lambda+\mu}{2}+\tau\frac{\lambda-\mu}{2} \ .
\]
Notice that this number can be immediately recovered by simply looking at the trace of $A$, namely
\begin{align*}
    \trace(A)&=(e^{\frac{\lambda}{2}}+e^{-\frac{\lambda}{2}})e^{+}+(e^{\frac{\mu}{2}}+e^{-\frac{\mu}{2}})e^{-}\\
    &=2\cosh\Big(\frac{\lambda}{2}\Big)e^{+}+2\cosh\Big(\frac{\mu}{2}\Big)e^{-}=2\cosh\Big(\frac{\lambda}{2}e^{+}+\frac{\mu}{2}e^{-}\Big)
\end{align*}
where the last step can be formally justified by considering the definition of the hyperbolic cosine as power series. We then conclude that
\[
    \ell_{\B}(A)=2\arccosh\Big(\frac{\trace(A)}{2}\Big) \ .
\]
We remark that for loxodromic isometries, the $\B$-length is always a space-like element of $\B$ with positive real part.

\subsubsection{Semi-loxodromic isometries} Up to conjugation, a semi-loxodromic isometry is of the form
\[
 B^{+}=\begin{pmatrix} e^{\frac{\lambda}{2}} & 0 \\ 0 & e^{-\frac{\lambda}{2}} \end{pmatrix}e^{+}+\begin{pmatrix} 1 & b \\ 0 & 1 \end{pmatrix}e^{-}  \ \ \ \text{or} \ \ \ B^{-}=\begin{pmatrix} 1 & b' \\ 0 & 1 \end{pmatrix}e^{+}+\begin{pmatrix} e^{\frac{\lambda}{2}} & 0 \\ 0 & e^{-\frac{\lambda}{2}} \end{pmatrix}e^{-}
\]
for some $b,b' \in \R\setminus \{0\}$ and $\lambda >0$. They have two fixed points in $\Pp\B^{1}$ ($[1,0]$ and $[e^{-}, e^{+}]$ in the first case, and $[1,0]$ and $[e^{+},e^{-}]$ in the second case) which lie on a light-like line in $\Pp\B^{1}$. In analogy with the loxodromic case, we define the $\B$-length of a semi-loxodromic isometry as
\[
    \ell_{\B}(B^{\pm}):=2\arccosh\Big(\frac{\trace(B^{\pm})}{2}\Big)=\frac{\lambda}{2}\pm \tau \frac{\lambda}{2} \ .
\]
Notice that the $\B$-length of a semi-loxodromic isometry is never invertible in $\B$ and its real part is always positive.

\subsubsection{Parabolic isometries}Up to conjugation, a parabolic isometry is given by
\[
    P=\begin{pmatrix} 1 & a \\ 0 & 1 \end{pmatrix}e^{+}+\begin{pmatrix} 1 & b \\ 0 & 1 \end{pmatrix}e^{-} 
\]
for some $a,b \in \R\setminus \{0\}$. Parabolic isometries have a unique fixed points in $\Pp\B^{1}$ and, analogously to parabolic isometries of the hyperbolic space, they leave \emph{horotori} invariant: namely, for every $c>0$, the set $H_{c}=\{ x_{1}+x_{2}\tau+cj \ | \ x_{1},x_{2} \in \R\} \subset V$, compactifies to a torus in $X$ and is invariant under $P$ because
\[
    P(x_{1}+x_{2}\tau+cj)=x_{1}+\frac{a}{2}+\frac{b}{2}+\Big(x_{2}+\frac{a}{2}-\frac{b}{2}\Big)\tau+cj \ .
\]
We can extend the formula for the $\B$-length to parabolic isometries obtaining 
\[
    \ell_{\B}(P):=2\arccosh\Big(\frac{\trace(P)}{2}\Big)=0 \ .
\]

\subsection{Cross ratio in $\Pp\B^{1}$} The classical definition of cross ratio for points in a projective line can also be adapted to $\Pp\B^{1}$, but some attention is needed from the fact that more points lie at infinity.

\begin{defi}We say that two points $p_{1},p_{2} \in \Pp\B^{1}$ are in space-like position if they can be joined by a space-like geodesic in $X$.
\end{defi}

If $p_{1}=[z_{1},1]$ and $p_{2}=[z_{2},1]$ are in particular points of $\B$, Proposition \ref{prop:spacelikegeo} ensures that $p_{1}$ and $p_{2}$ are in space-like position if and only if $z_{1}-z_{2} \in \B^{*}$. Notice that $z_{1}-z_{2}$ is the determinant of the two-by-two matrix with columns $p_{1}$ and $p_{2}$. Moreover, if we change representatives of $p_{1}$ and $p_{2}$ the determinant is multiplied by a $\lambda \in \B^{*}$, and if we act by an element $A \in \PSL(2, \B)$, the new matrix with columns $Ap_{1}$ and $Ap_{2}$ is still invertible. Therefore, in general, two points $p_{1}=[x_{1},y_{1}]$ and $p_{2}=[x_{2},y_{2}]$ are in space-like position if and only if $x_{1}y_{2}-x_{2}y_{1}$ is invertible in $\B$. 

\begin{defi}\label{def:crossratio} Let $p_{1},p_{2},p_{3},p_{4} \in \Pp\B^{1}$ be pairwise in space-like position. We define the cross ratio as
\[
    \cro(p_{1},p_{2},p_{3},p_{4})= \frac{(x_{2}y_{1}-x_{1}y_{2})(x_{4}y_{3}-x_{3}y_{4})}{(x_{1}y_{4}-x_{4}y_{1})(x_{2}y_{3}-x_{3}y_{2})}\ ,
\]
where $p_{i}=[x_{i},y_{i}]$.
\end{defi}

The usual proof that the cross ratio does not depend on the choice of the representatives and is invariant under the action of $\PSL(2,\B)$ adapts verbatim to this setting. We remark that the cross ratio has been chosen so that
\[
  \cro([1,0], [-1,1],[0,1],[z,1])=z \ .   
\]

\section{Augmented deformation space}\label{sec:GHMC}
Let $S$ be a surface of genus $g$ (possibly with punctures) and negative Euler characteristic. In this section we define (regular) globally hyperbolic maximal anti-de Sitter structures on $S \times \R$ focusing, in particular, on the properties of the holonomy of peripheral curves.

\subsection{Globally hyperbolic maximal anti-de Sitter structures} Let $M$ be a $3$-dimensional Lorentzian manifold locally isometric to anti-de Sitter space. We say that $M$ is Globally Hyperbolic (GH) if $M$ contains an embedded Cauchy surface, i.e. a surface that intersects any inextensible causal curve in exactly one point. Moreover, we say that $M$ is Maximal (M) if it is maximal by isometric embeddings, in the sense that if $\varphi:M \rightarrow M'$ is an isometric embedding between globally hyperbolic Lorentzian manifolds that sends Cauchy surfaces to Cauchy surfaces then $\varphi$ is a global isometry. Globally hyperbolicity forces the manifold to be diffeomorphic to a product $S \times \R$ (\cite{MR0270697}). However, once the topological type is fixed, different globally hyperbolic maximal anti-de Sitter structures can be put on the same topological manifold and we can define a deformation space of GHM anti-de Sitter structures over $S\times \R$ as
\[
    \GH(S)=\Bigg\{ \mu=(f,M)  \ \Big| \ \begin{tabular}{cc}
        \text{$M$ is a GHM anti-de Sitter manifold}  \\
         \text{$f: S \times \R \rightarrow M$ is a diffeomorphism}
    \end{tabular}\Bigg\} \Big/ \sim \ ,
\]
where $(f,M) \sim (f',M')$ if and only if $f'\circ f^{-1}$ is homotopic to an isometry from $M$ to $M'$. The moduli space is then the quotient 
\[
    \mathcal{MGH}(S)=\GH(S)/\mathrm{Diffeo}(S\times \R) \ .
\]
Mess (\cite{Mess}) studied the deformation space of Cauchy-compact (C) globally hyperbolic maximal anti-de Sitter manifolds, meaning that he assumed the Cauchy surface $S$ to be closed. In that case, he proved that, if $S$ has genus at least $2$, the holonomy representation $\hol: \pi_{1}(S) \rightarrow \PSL(2,\B)\cong \dPSL$ uniquely determines the structure, and, in particular, found a homeomorphism
\[
    \GH(S) \cong \T(S)\times \T(S)
\]
between the deformation space of GHMC anti-de Sitter structures over $S\times \R$ and two copies of the Teichm\"uller space of $S$ obtained by projecting the holonomy onto each $\PSL(2,\R)$-factor. As a consequence, the holonomy maps every simple closed curve on $S$ to a loxodromic isometry of $X$. \\

More recently, we extended this theory to include non-compact Cauchy surfaces. We introduced (\cite{Tambu_regularAdS}) \emph{regular} GHM anti-de Sitter structures on $S\times \R$, where now $S$ is allowed to have a finite number of punctures, whose holonomy sends any non-peripheral simple closed curve to a loxodromic isometry and peripheral elements to either loxodromic, semi-loxodromic or parabolic. However, in contrast with the case of compact Cauchy surfaces, the holonomy does not determine the structure uniquely. We also need to consider the developing map $\dev: \tilde{M} \rightarrow X$ which identifies the universal cover of $M$ with the domain of dependence of $\D(\Gamma)$ of a curve $\Gamma$ on the boundary at infinity of anti-de Sitter space. It turns out that the curve $\Gamma$, together with the holonomy representation, is sufficient to reconstruct the manifold (\cite{Tambu_regularAdS}). For GHMC anti-de Sitter structures, the curve $\Gamma$ coincides with the limit set of the holonomy representation and is an acausal topological circle with the property that any pair of distinct points is in space-like position. When punctures are allowed, however, and the holonomy of peripheral elements is not always parabolic, the limit set of the holonomy is a Cantor set and globally hyperbolic maximal anti-de Sitter structures with given holonomy are in one-to-one correspondence with achronal equivariant completions of the limit set to a topological circle (\cite{BonSchlGAFA2009}). Regular GHM anti-de Sitter structures correspond to completions of the limit set following this set of rules:
\begin{itemize}
    \item if the holonomy of a peripheral element $\gamma$ is semi-loxodromic, then $\hol(\gamma)$ has two fixed points in $\Pp\B^{1}$ that lies on a light-like segment. We complete the limit set by adding such light-like segments equivariantly.
    \item if the holonomy of a peripheral element $\gamma$ is loxodromic, then $\hol(\gamma)$ leaves four points fixed in $\Pp\B^{1}$ and has a space-like axis. We complete the limit set by joining the endpoints of the axis with a light-like sawtooth, i.e. a "vee" made up of two consecutive light-like segments sharing as a common endpoint either of the other two fixed points of $\hol(\gamma)$. We can distinguish the two cases by looking at the time-orientation of the sawtooth: we say that it is of type $+1$ if it is future-directed, and of type $-1$ if past-directed.
\end{itemize}

\noindent The developing pair $(\dev_{\mu}, \hol_{\mu})$ of a GHM anti-de Sitter structure $\mu$ defines a natural topology on $\GH(S)$: let $\C(\Ein^{1,1})$ denote the space of closed sets in the boundary at infinity of anti-de Sitter space endowed with the Hausdorff topology. We have a natural injective map
\begin{align*}
    \Pi: \GH(S)&\rightarrow (\Hom(\pi_{1}(S), \PSL(2,\B))\times \C(\Ein^{1,1}))/ \PSL(2,\B) \\
    \mu &\mapsto (\hol_{\mu}, \partial_{\infty}(\dev_{\mu}))
\end{align*}
where $\PSL(2,\B)$ acts by conjugation on $\Hom(\pi_{1}(S), \PSL(2,\B))$ and by conformal transformations on the Einstein Universe. We endow $\GH(S)$ with the topology that makes $\Pi$ a homeomorphism onto its image. If $S$ is not closed, we endow the deformation space $\GH(S)^{reg}\subset \GH(S)$ of regular GHM anti-de Sitter structures with the subspace topology. We remark that, under the identification $\PSL(2,\B)\cong \dPSL$, the holonomy of a regular GHM anti-de Sitter structure corresponds to a pair of faithful and discrete representations and any such pair can be realised (\cite{Tambu_regularAdS}).\\

There is also another, more analytic way, of describing the deformation space of GHM anti-de Sitter structures. In the closed case, this is due to Krasnov and Schlenker (\cite{Schlenker-Krasnov}) who proved that $\GH(S)$ can be parameterised by the cotangent bundle of the Teichm\"uller space of $S$. In fact, they showed that, given a hyperbolic metric $h\in \T(S)$ and a holomorphic quadratic differential $q$ on $(S,h)$, there is a unique equivariant maximal embedding $\tilde{\sigma}:\tilde{S}\rightarrow X$ with induced metric
conformal to $h$ and second fundamental form given by the real part of $q$. The quotient of the domain of dependence of the boundary at infinity of $\tilde{\sigma}(\tilde{S})$ by the action of $\pi_{1}(S)$ gives the desired GHMC anti-de Sitter manifold. \\
In the non-compact case, an analogous result was proved in \cite{Tambu_regularAdS}: following a similar strategy we showed that there is a homeomorphism 
\[
    \Psi: \GH(S)^{reg} \rightarrow \mathcal{MQ}_{\leq 2}(S)
\]
where $\mathcal{MQ}_{\leq 2}(S)$ denotes the bundle over Teichm\"uller space of meromorphic quadratic differentials on $S$ with poles of order at most $2$ at the punctures. This description will be useful in order to introduce a natural topology on the augmented deformation space.

\subsection{The augmented deformation space} The augmented deformation space of GHMC anti-de Sitter structures has the purpose of describing all possible ways a sequence of (regular) GHMC anti-de Sitter structures can degenerate on the complement of a multi-curve. For this reason it is naturally stratified, and each stratum depends on the choice of a multi-curve on the surface. Let $S$ be a closed surface of genus $g \geq 2$.

\begin{defi} A multi-curve $\D$ is a collection of simple closed curves on $S$ that are pairwise disjoint, non-homotopic, non-contractible and non-peripheral. We allow $\mathcal{D}=\emptyset$ as a multi-curve. A pair of pants is a maximal multi-curve. 
\end{defi}

Let $\D$ be a multi-curve in $S$ and let $S_{1}, \cdots S_{k}$ be the connected components of $S\setminus \D$. Consider a curve $\gamma \in \D$ and choose an orientation of $\gamma$. Let $S_{i}$ and $S_{j}$ be the connected components of $S\setminus \D$ that lie on the left and on the right of $\gamma$ respectively (we allow for the possibility that $S_{i}=S_{j}$ if $\gamma$ is non-separating). The curve $\gamma$ determines two elements $[\gamma_{i}] \in \pi_{1}(S_{i})$ and $[\gamma_{j}] \in \pi_{1}(S_{j})$. 

\begin{defi} A tuple $(\mu_{1}, \dots , \mu_{k}) \in \prod_{j=1}^{k}\GH(S_{j})^{reg}$ is compatible across $\gamma$ if $\hol_{\mu_{i}}([\gamma_{i}])$ is conjugated to $\hol_{\mu_{j}}([\gamma_{j}])$ and, if they are loxodromic, the sawteeth in the boundary at infinity of the developing maps $\dev_{\mu_{i}}$ and $\dev_{\mu_{j}}$ are of the same type.
\end{defi}

\begin{notation} When $\mu$ is compatible across $\gamma$, we will simply use the notation $\hol_{\mu}(\gamma)$ to mean $\hol_{\mu_{i}}([\gamma])$. In fact, we will be interested only in quantities that are invariant under conjugation, so the choice of the connected component containing $\gamma$ will not matter. We use a similar notation, if the curve $\gamma'$ is entirely contained in one connected component.
\end{notation}

\begin{defi} The \emph{augmented} deformation space of globally hyperbolic maximal anti-de Sitter structures over $S \times \R$ is 
\[
    \GH(S)^{aug}=\bigcup_{\D \text{\ multicurve}} \GH(S,\D)^{reg} \ ,
\]
where $\GH(S,\D)^{reg}$ denotes the space of tuples of GHM anti-de Sitter structures that are compatible across every curve $\gamma \in \D$. 
\end{defi}

A natural way of introducing a topology in this space is by noticing that we can identify $\GH(S)^{aug}$ with the bundle $\mathcal{V}(S)$ of meromorphic quadratic differentials over the augmented Teichm\"uller space $\T(S)^{aug}$ of $S$. We are now going to describe how this can be accomplished. Let $\D$ be a multi-curve in $S$ and let $S_{1}, \dots, S_{k}$ be the connected components of $S\setminus \D$. We denote by $\T(S,\D)\subset \T(S)^{aug}$ the set of all marked complete hyperbolic metrics on $S \setminus \D$. Let $\mathcal{V}_{\D}(S) \subset \mathcal{V}(S)$ be the sub-bundle of meromorphic quadratic differentials over $\T(S,\D)$ with poles of order at most $2$ at the punctures. An element of $\mathcal{V}_{\D}(S)$ is a $2k$-uple $(h_{1}, \dots, h_{k}, q_{1}, \dots, q_{k})$, where $h_{j}$ is a marked complete hyperbolic metric and $q_{j}$ is a meromorphic quadratic differential on $S_{j}$ with poles of order at most $2$ at the punctures. There is also a compatibility condition on the residue (i.e. the coefficient of the term $z^{-2}$ in a Laurent expansion around a pole) of the quadratic differentials. If $\gamma \in \D$ and $S_{i}$ and $S_{j}$ are the (possibly coincident) connected components of $S\setminus \D$ bounding $\gamma$, let $p_{i}$ and $p_{j}$ be the punctures on $S_{i}$ and $S_{j}$ corresponding to $\gamma$. Then the residue of $q_{i}$ and $q_{j}$ at the punctures $p_{i}$ and $p_{j}$ must coincide. It follows from (\cite{Tambu_regularAdS}) that $\mathcal{V}_{\D}$ is in bijection with $\GH(S,\D)$. Therefore, we have a one-to-one correspondence
\[
    \Psi_{\D}: \mathcal{V}_{\D}(S) \rightarrow \GH(S, \D)^{reg}
\]
and we define the topology on the augmented deformation space of GHM anti-de Sitter structures on $S\times \R$ that makes $\Psi^{aug}:=\cup_{\D}\Psi_{\D}: \mathcal{V}(S)\rightarrow \GH(S)^{aug}$ a homeomorphism. The following remarks and lemma show that this is a reasonable topology.

\begin{oss}If $\mathcal{D}=\emptyset$, then $\T(S,\D)=\T(S)$ and the map $\Psi_{D}$ identifies $\GH(S)$ with the cotangent bundle to Teichm\"uller space of $S$, so this stratum inherits the standard topology on $\GH(S)$.
\end{oss}

\begin{oss}\label{rmk:degeneration} By \cite[Theorem A]{Tambu_degeneration} this topology has the property that if a sequence $\mu_{n} \in \GH(S)$ converges to $\mu_{\infty}=(\mu_{\infty, 1}, \dots, \mu_{\infty, k}) \in \GH(S, \D)^{reg}$, then $\hol_{\mu_{n}}$ restricted to each $\pi_{1}(S_{i})$ (where $S_{1}, \dots, S_{k}$ are the connected components of $S\setminus \D$) converges to $\hol_{\mu_{\infty, i}}$ and the curves $\partial_{\infty}(\dev_{\mu_{n}})$ converge to $\partial_{\infty}(\dev_{\mu_{\infty, i}})$, up to the $\PSL(2,\B)$-action. We recall that this means that we can find a sequence of isometries $A_{n} \in \PSL(2,\B)$ such that $A_{n}(\hol_{\mu_{n}})_{|_{\pi_{1}(S_{i})}}A_{n}^{-1}$ converges to $\hol_{\mu_{\infty, i}}$ as representations and $A_{n}\partial_{\infty}(\dev_{\mu_{n}})$ converges to $\partial_{\infty}(\dev_{\mu_{\infty, i}})$ in the Hausdorff topology. Although Theorem A in the reference above is stated only for GHMC anti-de Sitter structure, the same proof applies to regular GHM anti-de Sitter structures, thus the same conclusion holds for sequences $\mu_{n} \in \GH(S,\D')$ with $\emptyset\neq \D'\subset \D$.
\end{oss}

\begin{lemma}\label{lm:degeneration2}Let $\D' \subset \D$ be multi-curves on $S$. Let $S_{1}, \dots, S_{k}$ be the connected components of $S\setminus \D'$ and let $\D_{i}$ be the set of curves of $\D$ contained in $S_{i}$. Let $S_{i, 1}, \dots, S_{i, j_{i}}$ be the connected components of $S_{i}\setminus \D_{i}$. If $\hol_{\mu_{n,i}}$ restricted to each $S_{i,j}$ converges to $\hol_{\mu_{\infty, i,j}}$ and the curves $\partial_{\infty}(\dev_{\mu_{n,i}})$ converge to $\partial_{\infty}(\dev_{\mu_{\infty,i,j}})$, then $\mu_{n}=(\mu_{n,1}, \dots \mu_{n,k}) \in \GH(S,\D')$ converges to $\mu_{\infty}=(\mu_{\infty,1,1}, \dots, \mu_{\infty, k,j_{k}}) \in \GH(S,D)$ in $\GH(S)^{reg}$.
\end{lemma}
\begin{proof} Choose a base point $p_{i,j}$ on $S_{i,j}$ that gives an injection of $\pi_{1}(S_{i,j})$ into $\pi_{1}(S_{i})$. Let $\tilde{\sigma_{n,i}}: \tilde{S}_{i} \rightarrow X$ be the $\hol_{\mu_{n,i}}$-equivariant maximal embedding into anti-de Sitter space bounding $\partial_{\infty}(\dev_{\mu_{n,i}})$. By assumption, the sequence of curves $\partial_{\infty}(\dev_{\mu_{n,i}})$ converges in the Hausdorff topology to $\partial_{\infty}(\dev_{\mu_{\infty,i,j}})$. Then the maximal surfaces $\tilde{\sigma_{n,i}}(\tilde{S_{i}}$ converge smoothly on compact sets to the maximal surface bounding $\partial_{\infty}(\dev_{\mu_{\infty,i,j}})$ (cfr. \cite[Proposition 4.6]{Tambu_poly}), which, by uniqueness (cfr. \cite[Lemma 4.2]{Tambu_poly}), is $\hol_{\mu_{\infty,i,j}}$-equivariant. This implies the convergence of the embedding data of the maximal surfaces. 
\end{proof}

\section{Fenchel-Nielsen coordinates}\label{sec:coord}
The first step in defining Fenchel-Nielsen coordinates is to parameterise representations of a pair of pants into $\PSL(2,\B)$. Let \[
    \Gamma_{0,3}=<r,s,t \ | \ tsr=e >
\]
be the fundamental group of a pair of pants $P$, where each generator corresponds to an oriented boundary curve so that $P$ lies on the left of each of them. We want to determine all possible representations of $\Gamma_{0,3}$ into $\PSL(2,\B)$, up to conjugation, under the assumption that the holonomy of a peripheral curve is loxodromic, semi-loxodromic or parabolic. We will call such representations \emph{admissible}. Recalling the cassifications of the isometries of anti-de Sitter space (see Section \ref{sec:model}), an admissible representation of $\Gamma_{0,3}$ into $\PSL(2,\B)$ corresponds to a pair $(\rho^{+},\rho^{-})$ of discrete and faithful representations into $\PSL(2,\R)$. \\

We introduce the following notation: we define $\C^{+}=\{ z \in \B \ | \ \Ree(z)>0, \ |z|^{2}>0 \}$ the positive cone of space-like elements of $\B$ with positive real part. Note that the $\B$-lengths of loxodromic, semi-loxodromic and parabolic isometries take value in $\overline{\C^{+}}$. 

\begin{prop}There is a homeomorphism
\[
    F:\Hom^{adm}(\Gamma_{0,3}, \PSL(2,\B))/\PSL(2,\B) \rightarrow \overline{\C^{+}}^{3}
\]
obtained by associating to each admissible representation $\rho$ the $\B$-lengths $\ell_{\B}(\rho(r))$, $\ell_{\B}(\rho(s))$ and $\ell_{\B}(\rho(t))$ of the generators.
\end{prop}
\begin{proof} Classical hyperbolic geometry tells us that a faithful and discrete representation $\rho^{+}$ of $\Gamma_{0,3}$ into $\PSL(2,\R)$ is uniquely determined, up to conjugation, by the non-negative real numbers
\[
    \Bigg(2\arccosh\Big(\frac{\trace\big(\rho^{+}(r)\big)}{2}\Big), 2\arccosh\Big(\frac{\trace\big(\rho^{+}(s)\big)}{2}\Big), 2\arccosh\Big(\frac{\trace\big(\rho^{+}(t)\big)}{2}\Big)\Bigg) \ .
\]
Therefore, a pair of faithful and discrete representations $(\rho^{+}, \rho^{-})$ is uniquely determined by the six-tuple
\begin{equation}\label{eq:sixtuple}
    \Bigg(2\arccosh\Big(\frac{\trace\big(\rho^{\pm}(r)\big)}{2}\Big), 2\arccosh\Big(\frac{\trace\big(\rho^{\pm}(s)\big)}{2}\Big), 2\arccosh\Big(\frac{\trace\big(\rho^{\pm}(t)\big)}{2}\Big)\Bigg)\ .
\end{equation}
We denote by $\ell(\rho^{\pm}(x))=2\arccosh(\frac{\trace(\rho^{\pm}(x))}{2})$ for $x \in \Gamma_{0,3}$ the hyperbolic length of the element $x$. Recalling that the $\B$-length of an isometry $\rho(x)=\rho^{+}(x)e^{+}=\rho^{-}(x)e^{-}$ is given by
\[
    \ell_{\B}(\rho(x))=\frac{\ell(\rho^{+}(x))+\ell(\rho^{-}(x))}{2}+\tau\frac{\ell(\rho^{+}(x))-\ell(\rho^{-}(x))}{2} \ ,
\]
it is easy to check that the data of the $\B$-lengths $(\ell_{\B}(\rho(r)),\ell_{\B}(\rho(s)),\ell_{\B}(\rho(t)))$ is equivalent to the data of the six non-negative numbers in Equation (\ref{eq:sixtuple}), hence the map $F$ is injective. Surjectivity follows from the fact that any non-negative number can be realised as hyperbolic length of a boundary curve of a pair of pants. Continuity of $F$ and $F^{-1}$ are also classical facts from hyperbolic geometry and we leave the details to the reader.
\end{proof}

We now want to glue together representations of $\Gamma_{0,3}$ in order to obtain a representation of the fundamental group of a surface with boundary. There will not be a unique way of doing so, and the different outcomes will depend on a $\B$-twist-bend parameter. Let $\rho_{1},\rho_{2}: \Gamma_{0,3} \rightarrow \PSL(2,\B)$ be two admissible representations. We think of having two copies of $\Gamma_{0,3}$ with generators $r_{1},s_{1},t_{1}$ and $r_{2},s_{2},t_{2}$ respectively, defined as before. Algebraically, gluing $\rho_{1}$ and $\rho_{2}$ together along two generators, say $r_{1}$ and $r_{2}$ corresponds to constructing a representation $\rho: \Gamma_{0,4} \rightarrow \PSL(2,\B)$ where $\Gamma_{0,4}$ is the fundamental group of the four punctured sphere $S_{0,4}$ which can be obtained as amalgamated product of the two copies of $\Gamma_{0,3}$ that identifies the generator $r_{1}$ with $r_{2}^{-1}$. This can be done only if there exists a matrix $A \in \PSL(2,\B)$ such that $A\rho_{2}(r_{2})A^{-1}=\rho_{1}(r_{1})^{-1}$. Moreover, since in the new group $\Gamma_{0,4}$ the generator $r_{1}$ will represent a non-peripheral simple closed curve in $S_{0,4}$, the resulting representation can be the holonomy of a GHM anti-de Sitter structure on $S_{0,4}\times \R$ only if $\rho_{1}(r_{1})$ is loxodromic. Up to conjugation, we can assume that the axis of $\rho_{1}(r_{1})$ is the space-like geodesic $\gamma(t)=e^{t}j$ with repelling fixed point $[0,1]\in \Pp\B^{1}$ and attracting fixed point $[1,0] \in \Pp\B^{1}$, and that the repelling fixed point of $\rho_{1}(s_{1})$ is $[-1,1] \in \Pp\B^{1}$. Notice that the matrix $A$ such that $A\rho_{2}(r_{2})A^{-1}=\rho_{1}(r_{1})^{-1}$ is not uniquely determined, as we can replace $A$ with $ZA$ for any $Z \in \PSL(2,\B)$ in the centraliser of $\rho_{1}(r_{1})$. We can pick out a unique $A$ by requiring that the attracting fixed point of $A\rho_{2}(t_{2})A^{-1}$ is $[1,1] \in \Pp\B^{1}$. For every $Z \in \PSL(2,\B)$ in the centraliser of $\rho_{1}(r_{1})$, the representation $\rho_{Z}:\Gamma_{0,4} \rightarrow \PSL(2,\B)$ given by
\[
    \rho_{Z}(x)=\begin{cases}
                \rho_{1}(x) \ \ \ \ \ \ \ \ \ \ \ \ \ \ \ \ \ \ \ \ \ \ \ \ \ \ \ \ \text{if $x \in <r_{1},s_{1},t_{1}>$}\\
                ZA\rho_{2}(x)A^{-1}Z^{-1} \ \ \ \ \ \ \ \ \ \ \ \ \ \text{if $x \in <r_{2},s_{2},t_{2}>$}
            \end{cases}
\]
is well-defined and the two associated representations $\rho_{Z}^{\pm}$ into $\PSL(2,\R)$ are still faithful and discrete, thus $\rho_{Z}$ is the holonomy of a GHM anti-de Sitter structure on $S_{0,4}\times \R$ for any such choice of $Z$.
By our re-normalisation, any isometry $Z$ in the centraliser of $\rho_{1}(r_{1})$ can be written as
\[
    Z=\begin{pmatrix} e^{\frac{\lambda}{2}} & 0 \\ 0 & e^{-\frac{\lambda}{2}} \end{pmatrix}e^{+}+\begin{pmatrix} e^{\frac{\mu}{2}} & 0 \\ 0 & e^{-\frac{\mu}{2}} \end{pmatrix}e^{-}
\]
for some real numbers $\lambda$ and $\mu$. We know that $Z$ acts by translation on $\gamma(t)$ by a (signed) distance of $\frac{\lambda+\mu}{2}$ and by a rotation of hyperbolic angle $\frac{\lambda-\mu}{2}$ on the orthogonal to $\gamma(t)$. We define the $\B$-twist-bend parameter as
\[
    \tw_{\B}(\rho_{Z}(r_{1})):=\frac{\lambda+\mu}{2}+\tau\frac{\lambda-\mu}{2} \ .
\]
Notice that $\tw_{\B}(\cdot)$ can take any value in $\B$ and uniquely determines $Z$. It is also useful to notice that the attracting fixed point of $ZA\rho_{2}(t_{2})A^{-1}Z^{-1}$ is 
\begin{align*}
    Z[1,1]&=Z[e^{+}+e^{-}, e^{+}+e^{-}]=Z([1,1]e^{+}+[1,1]e^{-})\\
    &=[e^{\frac{\lambda}{2}},e^{-\frac{\lambda}{2}}]e^{+}+[e^{\frac{\mu}{2}}, e^{-\frac{\mu}{2}}]e^{-}\\
    &=[e^{\frac{\lambda}{2}}+e^{\frac{\mu}{2}}+\tau(e^{\frac{\lambda}{2}}-e^{\frac{\mu}{2}}), e^{-\frac{\lambda}{2}} +e^{-\frac{\mu}{2}}+\tau(e^{-\frac{\lambda}{2}}-e^{-\frac{\mu}{2}})]\\
    &=[e^{\lambda}+e^{\mu}+\tau(e^{\lambda}-e^{\mu}),2]=[e^{\lambda}e^{+}+e^{\mu}e^{-},1]
\end{align*}
This shows that we can express the $\B$-twist-bend parameter as a cross ratio: in fact, if we denote by $\rho_{Z}(x)^{\pm}$ the attracting and repelling fixed point of $\rho_{Z}(x)$, we have
\begin{align*}
    \log&(\cro(\rho_{Z}(r_{1})^{+}, \rho_{Z}(s_{1})^{-}, \rho_{Z}(r_{1})^{-}, \rho_{Z}(t_{2})^{+}))\\
    &=\log(\cro([1,0], [-1,1],[0,1],Z[1,1])) \\
    &=\log(e^{\lambda}e^{+}+e^{\mu}e^{-})=\lambda e^{+}+\mu e^{-}
    =\tw_{\B}(\rho_{Z}(r_{1})) \ .
\end{align*}

In order to obtain a representation of a closed surface group starting from a collection of representations of $\Gamma_{0,3}$ we also need to be able to glue two boundary curves of the same pair of pants. It is sufficient to describe how to obtain a representation of the fundamental group $\Gamma_{1,1}$ of a one-holed torus starting from a representation $\rho:\Gamma_{0,3}\rightarrow \PSL(2,\B)$. Let us suppose that we want to glue the curve represented by $r$ with that represented by $s$. As before, this can be done only if $\rho(r)$ and $\rho(s)$ are conjugated in $\PSL(2,\B)$ and are loxodromic. This means that we can find two loxodromic matrices $A,B \in \PSL(2,\B)$ such that $\rho(r)=A$ and $B\rho(s)B^{-1}=A^{-1}$. Notice that the element $B$ is not uniquely determined, because we can change $B$ by $ZB$ with $Z$ in the centraliser of $A$. Now, the group generated by $A$ and $ZB$ is the $HNN$-extension of the group generated by $\rho(s)$ and $\rho(r)$, thus it is isomorphic to $\Gamma_{1,1}$, where $ZB$ is the image of a simple loop intersecting the glued curve once. Given $\rho$, we obtain in this way a family of representations $\rho_{Z}:\Gamma_{1,1}\rightarrow \PSL(2,B)$ which are the holonomy of GHM anti-de Sitter structures. In order to define a $\B$-twist-bend parameter associated to $\rho_{Z}$, we first have to fix a canonical choice of $B$. To this aim, we consider the cross ratio
\[
    \cro(\rho(r)^{+}, B^{-1}Z^{-1}\rho(r)^{+}, \rho(r)^{-}, ZB\rho(r)^{-}) \ .
\]
Up to conjugation, we can assume that $\rho(r)^{+}=[1,0]\in \Pp\B^{1}$ and $\rho(r)^{-}=[0,1] \in \Pp\B^{1}$. Then, if $B^{-1}[1,0]=[x^{+}e^{+}+x^{-}e^{-}, y^{+}e^{+}+y^{-}e^{-}]$ and $B[0,1]=[z^{+}e^{+}+z^{-}e^{-}, w^{+}e^{+}+w^{-}e^{-}]$ and we write
\[
    Z=\begin{pmatrix} e^{\frac{\lambda}{2}} & 0 \\ 0 & e^{-\frac{\lambda}{2}} \end{pmatrix}e^{+}+\begin{pmatrix} e^{\frac{\mu}{2}} & 0 \\ 0 & e^{-\frac{\mu}{2}} \end{pmatrix}e^{-}
\]
we have that $B^{-1}Z^{-1}[1,0]=B[1,0]$ and $ZB[0,1]=[e^{\frac{+\lambda}{2}}z^{+}e^{+}+e^{+\frac{\mu}{2}}z^{-}e^{-}, e^{-\frac{\lambda}{2}}w^{+}e^{+}+e^{-\frac{\mu}{2}}w^{-}e^{-}]$
hence the cross ratio $\cro(\rho(r)^{+}, B^{-1}Z^{-1}\rho(r)^{+}, \rho(r)^{-}, ZB\rho(r)^{-})$ is equal to
\[
    \frac{(e^{\frac{\lambda}{2}}z^{+}e^{+}+e^{\frac{\mu}{2}}z^{-}e^{-})(y^{+}e^{+}+y^{-}e^{-})}{(e^{-\frac{\lambda}{2}}w^{+}e^{+}+e^{-\frac{\mu}{2}}w^{-}e^{-})(w^{+}e^{+}+w^{-}e^{-})}
    =\frac{z^{+}y^{+}}{w^{+}x^{+}}e^{\lambda}e^{+}+\frac{z^{-}y^{-}}{w^{-}x^{-}}e^{\mu}e^{-}
\]
and there is a unique choice of $\lambda$ and $\mu$, so that the above cross ratio equals $1$. Therefore, we can fix a canonical choice $B_{0}=ZB$ so that
\[
    \cro(\rho(r)^{+}, B_{0}^{-1}\rho(r)^{+}, \rho(r)^{-}, B_{0}\rho(r)^{-})=1
\]
and then for any $Z$ in the centraliser of $\rho(r)$ we define the $\B$-twist-bend parameter as
\begin{align*}
    \tw_{\B}(\rho_{Z}(r))&=\log(\cro(\rho(r)^{+}, B_{0}^{-1}Z^{-1}\rho(r)^{+}, \rho(r)^{-}, ZB_{0}\rho(r)^{-}))\\ &=\frac{\lambda+\mu}{2}+\tau\frac{\lambda-\mu}{2} \ ,
\end{align*}
if, as before, we write
\[
    Z=\begin{pmatrix} e^{\frac{\lambda}{2}} & 0 \\ 0 & e^{-\frac{\lambda}{2}} \end{pmatrix}e^{+}+\begin{pmatrix} e^{\frac{\mu}{2}} & 0 \\ 0 & e^{-\frac{\mu}{2}} \end{pmatrix}e^{-} \ .
\]
\\

We can now define Fenchel-Nielsen coordinates for the holonomy representation of a GHMC anti-de Sitter structure. Let $S$ be a closed, connected, oriented surface of genus $g\geq 2$. Fix a pants decomposition $\PP=\{\gamma_{1}, \cdots, \gamma_{3g-3}\}$ on $S$. Let $\rho:\pi_{1}(S)\rightarrow \PSL(2,\B)$ be the holonomy representation of a GHMC anti-de Sitter structure on $S\times \R$. The restriction of $\rho$ to each pair of pants gives a collection of representations of $\Gamma_{0,3}$ into $\PSL(2,\B)$, with the property that every peripheral element is sent to a loxodromic isometry. For any $\gamma \in \PP$ there is also a well-defined $\B$-twist bend parameter that encodes the gluing along a curve $\gamma$ as explained before. 
\begin{prop} The map
\begin{align*}
    \Phi_{\PP}:\GH(S) &\rightarrow (\C^{+}\times \B)^{3g-3}  \\
                \rho &\mapsto (\ell_{B}(\rho(\gamma_{1})), \tw_{\B}(\rho(\gamma_{1})), \dots, \ell_{B}(\rho(\gamma_{3g-3})), \tw_{\B}(\rho(\gamma_{3g-3})))
\end{align*}
is a homeomorphism.
\end{prop}
\begin{proof} The claim follows from Mess' parameterisation of $\GH(S)$ as $\T(S)\times \T(S)$ by noticing that, if $\rho=(\rho^{+},\rho^{-})$, then
\[
    \ell_{\B}(\rho(\gamma_{i}))=\frac{\ell(\rho^{+}(\gamma_{i}))+\ell(\rho^{-}(\gamma_{i}))}{2}+\tau \frac{\ell(\rho^{+}(\gamma_{i}))-\ell(\rho^{-}(\gamma_{i}))}{2} 
\]
and
\[
    \tw_{B}(\rho(\gamma_{i}))=\frac{\tw(\rho^{+}(\gamma_{i}))+\tw(\rho^{-}(\gamma_{i}))}{2}+\tau \frac{\tw(\rho^{+}(\gamma_{i}))-\tw(\rho^{-}(\gamma_{i}))}{2} \ , 
\]
where $\ell(\rho^{\pm}(\gamma_{i}))$ and $\tw(\rho^{\pm}(\gamma_{i}))$ are the classical Fenchel-Nielsen coordinates in Teichm\"uller space.
\end{proof}

We can generalise the above result to regular GHM anti-de Sitter structures. Suppose now that $S$ is a surface with $n$ punctures, genus $g$ and negative Euler characteristic. Fix a pair decomposition $\PP=\{\gamma_{1}, \dots, \gamma_{3g-3+n}\}$ of $S$ and let $\sigma_{1}, \dots \sigma_{n}$ be peripheral curves. Let $\hol_{\mu}: \pi_{1}(S)\rightarrow \PSL(2,\B)$ be the holonomy representation of a regular GHM anti-de Sitter structure $\mu$ on $S\times \R$. For every curve in $\PP$ we can define, as before, a $\B$-length taking value in $\C^{+}$ and a $\B$-twist-bend parameter. On the other hand, for each curve $\sigma_{i}$ we can define a $\B$-length which will now take value in $\overline{\C^{+}}$ because $\hol_{\mu}(\sigma_{i})$ may be semi-loxodromic or parabolic as well, but the $\B$-twist-bend is not defined. Moreover, since the holonomy does not determine the structure uniquely, we introduce a parameter $\epsilon_{\mu}(\sigma_{i})$ for every $i=1, \dots, n$ defined as follows:
\begin{itemize}
    \item $\epsilon_{\mu}(\sigma_{i})=0$ if $\hol_{\mu}(\sigma_{i})$ is semi-loxodromic or parabolic;
    \item $\epsilon_{\mu}(\sigma_{i})=+1$ if $\hol_{\mu}(\sigma_{i})$ is loxodromic and its attracting and repelling fixed points are connected in $\partial_{\infty}(\dev_{\mu_{i}})$ by a future-directed sawtooth;
    \item $\epsilon_{\mu}(\sigma_{i})=-1$ if $\hol_{\mu}(\sigma_{i})$ is loxodromic and its attracting and repelling fixed points are connected in $\partial_{\infty}(\dev_{\mu_{i}})$ by a past-directed sawtooth.
\end{itemize}

\begin{teo}\label{thm:coord}
The map
\begin{align*}
    \Phi_{\PP}^{reg}:\GH(S)^{reg} &\rightarrow (\C^{+}\times \B)^{3g-3+n}\times (\overline{\C^{+}}\times \R)^{n}\\
        \mu &\mapsto \prod_{i=1}^{3g-3+n}(\ell_{\B}(\hol_{\mu}(\gamma_{i})), \tw_{\B}(\hol_{\mu}(\gamma_{i}))) \prod_{i=1}^{n}(\ell_{\B}(\hol_{\mu}(\sigma_{i})), \delta_{\mu}(\sigma_{i})) \ ,
\end{align*}
where $\delta_{\mu}(\sigma_{i})=\epsilon_{\mu}(\sigma_{i})|\ell_{\B}(\hol_{\mu}(\sigma_{i}))|$ is a homeomorphism  onto the image.
\end{teo}
\begin{proof}Let $p=(x_{1}, \dots, x_{3g-3+n}, y_{1}, \dots, y_{3g-3+n}, z_{1}, \dots z_{n}, w_{1}, \dots w_{n})$ be a point in the target space such that $w_{i}^{2}=|z_{i}|^{2}$ for every $i=1, \dots, n$. We denote by $E$ the subset of points in $(\C^{+}\times \B)^{3g-3+n}\times (\overline{\C^{+}}\times \R)^{n}$ satisfying these relations. We are going to show that every point of $E$ can be uniquely realised as the image of a GHM anti-de Sitter structure under $\Phi_{\PP}^{reg}$, thus showing that $\Phi_{\PP}^{reg}$ is injective with image $E$. Notice, in fact, that the image of $\Phi_{\PP}^{reg}$ is necessarily contained in $E$. \\
Cutting the surface along all the curves $\gamma_{i}$ we obtain a collection of pairs of pants with boundary curves $\gamma_{i}$ or $\sigma_{i}$. The coordinates $x_{i}$ and $w_{i}$ assign the $\B$-lengths of the boundary curves of the pair of pants, hence there is a unique family of representations of $\Gamma_{0,3}$ into $\PSL(2,\B)$ realising them. The coefficients $y_{i}$ give a unique way of gluing these representations so that the $\B$-twist-bend parameters of the resulting representation $\rho$ along the curve $\gamma_{i}$ is precisely $y_{i}$, as explained previously in this section. Therefore, the holonomy $\rho$ is uniquely determined by $(x_{1}, \dots, x_{3g-3+n}, y_{1}, \dots, y_{3g-3+n}, z_{1}, \dots z_{n})$. Finally, looking at the sign of $w_{i}$ we can determine the boundary at infinity of the developing map, and thus the structure, uniquely. Namely, if $w_{i}=0$, then $z_{i}\in \partial\overline{\C^{+}}$ and $\rho(\sigma_{i})$ is semi-loxodromic or parabolic. Otherwise, $z_{i}$ is loxodromic and the sign of $w_{i}$ determines which type of sawtooth to use to complete the limit set of $\rho$ to an achronal topological circle.\\
\indent Let us check that the map $\Phi_{\PP}^{reg}$ is continuous. Let $\mu_{n}$ be a sequence of GHM anti-de Sitter structures converging to $\mu$. Then $\hol_{\mu_{n}}$ converges to $\hol_{\mu}$, up to conjugation, and $\partial_{\infty}(\dev_{\mu_{n}})$ converges to $\partial_{\infty}(\dev_{\mu})$ in the Hausdorff topology. This immediately implies that the $\B$-lengths and the $\B$-twist-bend parameters converge, as they depend continuously, by definition, on the holonomy representation. Let us now consider what happens to the coordinates $\delta_{\mu_{n}}(\sigma_{i})$. We know that $\hol_{\mu_{n}}(\sigma_{i})$ converges to $\hol_{\mu}(\sigma_{i})$. We distinguish two cases:
\begin{itemize}
    \item if $\hol_{\mu}(\sigma_{i})$ is semi-loxodromic or parabolic, then $\lim_{n \to \infty} |\ell_{\B}(\hol_{\mu_{n}}(\sigma_{i}))|=0$ and $\delta_{\mu}(\sigma_{i})=0$. Because
    \[
        \lim_{n \to \infty}\delta_{\mu_{n}}(\sigma_{i})=\lim_{n \to \infty} |\ell_{\B}(\hol_{\mu_{n}}(\sigma_{i}))|\epsilon_{\mu_{n}}(\sigma_{i})=0  \ , 
    \]
    we have that $\delta_{\mu_{n}}(\sigma_{i})$ converges to $\delta_{\mu}(\sigma_{i})$. 
    \item if $\hol_{\mu}(\sigma_{i})$ is loxodromic, then the convergence of the limit set in the Hausdorff topology implies that the sawteeth connecting the attracting and repelling fixed points of $\hol_{\mu_{n}}(\sigma_{i})$ are eventually of the same type. Therefore, $\epsilon_{\mu_{n}}(\sigma_{i})$ are eventually $+1$ or $-1$ independently of $n$, so $\delta_{\mu_{n}}(\sigma_{i})$ converges to  $\delta_{\mu}(\sigma_{i})$.
\end{itemize}
We are left to prove that $\Phi_{\PP}^{reg}$ is proper. Let $\mu_{n}$ be a sequence of GHM anti-de Sitter structures such that $p_{n}=\Phi_{\PP}^{reg}(\mu_{n})$ converges to $p$ in $E$. This immediately implies that $\hol_{\mu_{n}}$ converges to an admissible representation $\rho:\pi_{1}(S)\rightarrow \PSL(2,\B)$. In particular the limit set $\Gamma_{n}$ of $\hol_{\mu_{n}}$ converges to the limit set $\Gamma_{\rho}$ of $\rho$. Looking at the $w$-coordinates of $p$, we have a way of completing $\Gamma_{\rho}$ in a $\rho$-equivariant way to an achronal topological circle $\tilde{\Gamma}_{\rho}$. The pair $(\rho, \tilde{\Gamma}_{\rho})$ determines a regular GHM anti -de Sitter structure $\mu$ such that $\hol_{\mu}=\rho$ and $\partial_{\infty}(\dev_{\mu})=\tilde{\Gamma}_{\rho}$. To conclude, we only need to show that $\partial_{\infty}(\dev_{\mu_{n}})$ converges to $\tilde{\Gamma}_{\rho}$ in the Hausdorff topology. Consider a peripheral curve $\sigma_{i}$. We distinguish three cases:
\begin{itemize} 
    \item if $\rho(\sigma_{i})$ is loxodromic, then the attracting and repelling fixed point of $\rho(\sigma_{i})$ are the limit of the attracting and repelling fixed points of $\hol_{\mu_{n}}(\sigma_{i})$. Since the coordinates $w_{i,n}$ of $p_{n}$ are converging to the coordinate $w_{i}$ of $p$, the sawteeth joining the attracting and repelling fixed points of $\hol_{\mu_{n}}(\sigma_{i})$ are eventually of the same type, and converge to the sawtooth joining the attracting and repelling fixed point of $\rho(\sigma_{i})$ in $\Gamma_{\rho}$. 
    \item If $\rho(\sigma_{i})$ is semi-loxodromic or parabolic and $\hol_{\mu_{n}}(\sigma_{i})$ is eventually semi-loxodromic or parabolic, the convergence of their fixed points in $\Ein^{1,1}$ already implies the convergence of the completion of the limit set in the end corresponding to $\sigma_{i}$.
    \item If $\rho(\sigma_{i})$ is semi-loxodromic or parabolic and it is the limit of loxodromic isometries, then the corresponding sawteeth in $\partial_{\infty}(\dev_{\mu_{n}})$ collapse to a light-like segment (if $\rho(\sigma_{i})$ is semi-loxodromic) or to a point (if $\rho(\sigma_{i})$ is parabolic) because the fixed points of $\hol_{\mu_{n}}(\sigma_{i})$ collide in pairs to the two fixed points of $\rho(\sigma_{i})$, if it is semi-loxodromic, and collapse to the unique fixed point of $\rho(\sigma_{i})$, if it is parabolic.
\end{itemize}
\end{proof}

This result can also be extended to regular globally hyperbolic maximal anti-de Sitter structures that are compatible across a multi-curve $\D$ by the fact that an element of $\GH(S,\D)^{reg}$ is simply a tuple of regular GHM anti-de Sitter structures with a compatibility condition on every $c \in \D$.

\begin{cor}\label{cor:augmented}Let $\D$ be a multi-curve on $S$ and let $S_{1}, \dots, S_{k}$ be the connected components of $S\setminus \D$ ad let $g_{j}$ and $n_{j}$ be the genus and the number of punctures of the surface $S_{j}$ for $j=1, \dots k$. Fix pants decomposition $\PP_{j}$ on every $S_{j}$. Then the map
\begin{align*}
    \prod_{j=1}^{k}\Phi_{\PP_{j}}^{reg}: \GH(S,\D)^{reg} & \rightarrow \prod_{j=1}^{k}(\C^{+}\times \B)^{3g_{j}-3+n_{j}}\times (\overline{\C^{+}} \times \R)^{n_{j}}\\
        (\mu_{1}, \dots, \mu_{k}) &\mapsto (\Phi_{\PP_{1}}(\mu_{1}), \dots, \Phi_{\PP_{k}}(\mu_{k}))
\end{align*}
is a homeomorphism onto its image.
\end{cor}

\section{Coordinates on the augmented moduli space}\label{sec:mainthm}
In this section we are going to see how to continuously extend the coordinates in Theorem \ref{thm:coord} to the augmented deformation space. For simplicity, we assume that $S$ is closed of genus $g\geq 2$. Recall that the augmented deformation space $\GH(S)^{aug}$ encodes how a structure can degenerate along a multi-curve $\D$. It turns out that we are able to extend the coordinates of Theorem \ref{thm:coord} if we quotient out the action of a subgroup $G_{\D}$ of the mapping class group generated by Dehn-twists about the simple closed curves in $\D$. This should not come as a surprise because in the classical Teichm\"uller theory a similar phenomenon occurs as well, when trying to extend the Fenchel-Nielsen coordinates to the augmented Teichm\"uller space. 

\begin{teo}\cite{Hubbard_Koch}\label{thm:augteich}. Let $\T(S)^{aug,\D}$ be the Teichm\"uller space augmented along the curves in $\D$ only. Choose a pants decomposition $\PP\supset \D$ of $S$, and let $\ell_{i}$ and $\theta_{i}$ denote the length and twist coordinates of curves in $\PP$, where $\theta_{i}$ are re-normalised so that Dehn twists act by $\theta_{i}\mapsto \theta_{i}+2\pi$. Then, the quotient $\T(S)^{aug,\D}/G_{\D}$ has global coordinates given by
\[
    \Bigg(\prod_{j=1}^{m}(\ell_{j}\cos(\theta_{j}), \ell_{j}\sin(\theta_{j}))\Bigg) \times \Bigg(\prod_{j=m+1}^{3g-3}(\ell_{j}, \theta_{j})\Bigg) \in \R^{2m}\times (\R^{+}\times \R)^{3g-3} \ ,
\]
where the first $m$ coordinates are related to curves in $\D$.
\end{teo}

Let us first describe the action of a Dehn-twist in our setting. Fix a pants decomposition $\PP\supset \D$ on the surface $S$. A Dehn-twist about a curve $\gamma \in \D$ acts on $\Phi_{\PP}(\GH(S))$ by changing the associated $\B$-twist-bend parameters as follows
\[
    \tw_{\B}(\hol_{\mu}(\gamma)) \mapsto \tw_{\B}(\hol_{\mu}(\gamma))+\ell_{\B}(\hol_{\mu}(\gamma)) \ ,
\]
and does not affect the other coordinates. It is then convenient to re-normalise the $\B$-twist-bend parameter by defining
\[
    \theta_{\B}(\hol_{\mu}(\gamma)):=2\pi \frac{\tw_{\B}(\hol_{\mu}(\gamma))}{\ell_{\B}(\hol_{\mu}(\gamma))}
\]
so that the Dehn-twist about $\gamma$ acts as
\[
    \theta_{\B}(\hol_{\mu}(\gamma)) \mapsto \theta_{\B}(\hol_{\mu}(\gamma))+2\pi \ .
\]
In particular, the imaginary part of $\theta_{\B}(\hol_{\mu}(\gamma))$ is invariant under the action of a Dehn-twist. The following change of coordinates will also be useful:

\begin{lemma}\label{lm:changecoord} The map $H:\C^{+} \times S^{1} \times \R  \rightarrow \R\times (\R^{2}\setminus \{0\}) \times \R$ defined by $H(a+\tau b, e^{i\theta}, c)$
\[
=\Bigg(b, |a+\tau b|\sqrt{1-\tanh^{2}(c)}\cos(\theta), |a+\tau b|\sqrt{1-\tanh^{2}(c)}\sin(\theta), |a+\tau b|\tanh(c)\Bigg)
\]
is a homeomorphism.
\end{lemma} 
\begin{proof}The map $H$ is clearly continuous and its inverse is 
\[
    H^{-1}(x,y,z,w)=\Bigg(\sqrt{x^{2}+y^{2}+z^{2}+w^{2}}+\tau y, \frac{y+iz}{\sqrt{y^{2}+z^{2}}} ,\arctanh\Big(\frac{w}{y^{2}+z^{2}+w^{2}}\Big)\Bigg) \ ,
\]
which is well-defined and continuous because $(y,z)\in \R^{2}\setminus \{0\}$\ .
\end{proof}
 
\begin{defi}Let $\D'\subset \D$ be multi-curves in $S$ and let $\mu \in \GH(S,\D')$. Given $\gamma \in \D$, we define
\[
  \mathcal{H}(\mu(\gamma))=\begin{cases}
    H\big(\ell_{\B}(\hol_{\mu}(\gamma)), e^{i\Ree(\theta_{\mu}(\gamma))}, \Imm(\theta_{\mu}(\gamma))\big) \  \ \ \ \ \ \ \ \text{if $\gamma \notin \D'$} \\
    (\Imm(\ell_{\B}(\hol_{\mu}(\gamma))), 0, 0, \delta_{\mu}(\gamma)) \ \ \ \ \ \ \ \ \ \ \ \ \ \ \  \ \ \ \ \text{if $\gamma \in \D'$}
    \end{cases}
\]
\end{defi}

We can now prove the main theorem:

\begin{teo}Let $\mu \in \mathcal{GH}(S)^{aug}$ and let $\D$ be a multi-curve so that $\mu \in \mathcal{GH}(S,\D)^{reg}$. Fix a pants decomposition $\PP \supset \D$ of $S$. Let $G_{\D}$ be the subgroup of the mapping class group generated by Dehn-twists about the simple closed curves in $\D$. Then
\[
    \GH(S)^{aug,\D}:=\bigcup_{\D'\subset \D}\GH(S,\D')^{reg}
\]
is an open subset of $\GH(S)^{aug}$ containing $\mu$ and invariant under $G_{\D}$. Moreover, the map $\Phi_{\PP, \D}: \GH(S)^{aug, \D}/G_{\D}\rightarrow (\C^{+}\times \B)^{3g-3-m}\times \R^{4m}$ defined by
\[
\mu \mapsto \prod_{\gamma \in \PP\setminus \D}(\ell_{\B}(\hol_{\mu}(\gamma)), \tw_{\B}(\hol_{\mu}(\gamma))) \prod_{\gamma \in \D}\mathcal{H}(\mu(\gamma)) \ ,
\]
where $m$ is the number of curves in $\D$, is a homeomorphism. 
\end{teo}
\begin{proof} We first verify that $\GH(S)^{aug, \D}$ is an open set. Let $\Psi^{aug}:\mathcal{V}(S)\rightarrow \GH(S)^{aug}$ be the bijection defined in Section \ref{sec:GHMC}. By definition of the topology in $\GH(S)^{aug}$, the set $\GH(S)^{aug, \D}$ is open if and only if it is the image of an open set under $\Phi^{aug}$. Now,
\[
    \GH(S)^{aug,\D}=\bigcup_{\D'\subset \D}\GH(S,\D')^{reg}=\bigcup_{\D' \subset \D}\Psi_{\D'}(\mathcal{V}_{\D'})=\Psi^{aug}\Big(\bigcup_{\D' \subset \D}\mathcal{V}_{\D'} \Big)
\]
and $\bigcup_{\D' \subset \D}\mathcal{V}_{\D'}$ is open because it is the bundle of meromorphic quadratic differentials of poles of order at most $2$ (with a compatibility condition) over $\T(S)^{aug,\D}$, which is open by Theorem \ref{thm:augteich}. \\
\indent The map $\Phi_{\PP,\D}$ is well-defined because a Dehn-twist about $\gamma$ leaves the imaginary part of $\theta_{\B}(\gamma)$ invariant and adds $2\pi$ to the real part of $\theta_{\B}(\gamma)$, so the value of  the function $H(\mu(\gamma))$ is invariant under the action of $G_{\D}$ and descends to the quotient. \\
\indent Let us show that $\Phi_{\PP, \D}$ is bijective. We enumerate the curves in $\D=\{\gamma_{1}, \dots, \gamma_{m}\}$. Let $p$ be a point in $(\C^{+}\times \B)^{3g-3-m}\times \R^{4m}$. We write 
\begin{equation}\label{eq:longcoord}
    p=(x_{1}, y_{1}, \dots, x_{3g-3+m}, y_{3g-3+m}, a_{1}, b_{1}, c_{1}, d_{1}, \dots, a_{m}, b_{m}, c_{m}, d_{m}) \ .
\end{equation}Let $\D'\subset \D$ consisting of all the curves $\gamma_{j} \in \D$ such that the index $j$ satisfies $b_{j}^{2}+c_{j}^{2}= 0$. If $p$ is in the image of $\Phi_{\PP,\D}$, it must necessarily come from a $\mu \in \GH(S,\D')$. Notice that the curves in $\PP \setminus \D'$ give pair of pants decompositions for each connecting component $S_{1}, \dots, S_{k}$ of $S\setminus \D'$, thus we can find such a $\mu=(\mu_{1}, \dots, \mu_{k})$ by considering the representations $\rho_{i}: \pi_{1}(S_{i})\rightarrow \PSL(2,\B)$ and the curves at infinity $\Gamma_{i} \subset \Ein^{1,1}$ for $i=1, \dots k$, with the following properties:
\begin{itemize}
    \item the $\B$-lengths of the curves in $\PP \setminus \D$ are given by the $x_{i}$-coordinates of $p$;
    \item the $\B$-twist-bend parameters of the curves in $\PP \setminus \D$ are given by the $y_{i}$-coordinates;
    \item the $\B$-lengths of the curve $\gamma_{i} \in \D \setminus \D'$ are given by the first coordinate of $H^{-1}(a_{i}, b_{i}, c_{i}, d_{i})$;
    \item the re-normalised $\B$-twist-bend parameter of the curve $\gamma_{i} \in \D \setminus \D'$ is given by $\theta_{i}+\tau w_{i}$ where $H^{-1}(a_{i}, b_{i}, c_{i}, d_{i})=(z_{i}, e^{i\theta_{i}}, w_{i})$. This determines the $\B$-twist-bend parameter of $\gamma_{i}$ only up to integer multiples of $2 \pi$. However, different choices will be related by Dehn-twists.
    \item the $\B$-lengths of the curves $\gamma_{j} \in \D'$ are given by $\sqrt{d_{j}^{2}-a_{j}^{2}}+\tau a_{j}$.
    Notice that $\gamma_{j}$ is sent to a loxodromic isometry if and only if $d_{j} \neq 0.$
    \item The limit set of each of the representations $\rho_{i}$ obtained from the above rules is completed, in a $\rho_{i}$-equivariant way, to an achronal topological circle $\Gamma_{i}$ by connecting points that lie on the same light-like segments and inserting past-directed or future-directed sawteeth joining the endpoints of $\rho_{i}(\gamma_{j})$ according to the sign of $d_{j}$. 
\end{itemize}
The quotient of the domain of dependence of $\Gamma_{i}$ by $\rho_{i}$ gives a regular GHM anti-de Sitter structure $\mu_{i}$ on $S_{i} \times \R$ for every $i=1, \dots, k$. It is clear from the definition of the map $\Phi_{\PP, \D}$ that, setting $\mu=(\mu_{1}, \dots, \mu_{k})$, we have $\Phi_{\PP,\D}([\mu])=p$. Moreover, such $\mu$ is unique up to Dehn-twists by Corollary \ref{cor:augmented}. In addition we remark that the holonomy and the developing maps of these structures $\mu_{i}$ depend continuously on the coefficients of $p$ by construction.\\
\indent Let us now show that the map $\Phi_{\PP, \D}$ is continuous. Let $\mu_{n} \in \GH(S)^{aug}$ converging to $\mu_{\infty} \in \GH(S)^{aug}$. Let $\D'\subset \D$ such that $\mu_{\infty} \in \GH(S,\D')$. If $\mu_{n}\in \GH(S,\D')$ for $n$ sufficiently large, then we already know that $\Phi_{\PP,\D}(\mu_{n})$ converges to $\Phi_{\PP,\D}(\mu)$ by Corollary \ref{cor:augmented}. Otherwise, since $\D'$ is a finite set, we can assume that there is a multi-curve $\D''\subsetneq \D'$ such that $\mu_{n}\in \GH(S, \D'')$ for all $n$. Let $S_{1}, \dots, S_{k}$ be the connected components of $S\setminus \D''$. We denote by $\D'_{i}$ the set of curves in $\D'$ that are contained in $S_{i}$. Let $S_{i,1}, \dots, S_{i,j_{i}}$ be the connected components of $S_{i}\setminus \D'_{i}$. Let us write $\mu_{n}=(\mu_{n,i}, \dots, \mu_{n,k})$ where $\mu_{n,i}$ is a regular GHM anti-de Sitter structure on $S_{i}\times \R$, and $\mu_{\infty}=(\mu_{\infty,1,1}, \dots, \mu_{\infty,1, j_{1}}, \dots, \mu_{\infty, k,1}, \dots \mu_{\infty, k, j_{k}})$, where $\mu_{\infty,i,j_{i}}$ is a regular GHM anti-de Sitter structure on $S_{i,j}\times \R$. Let us denote by $\rho_{n,i,j}$ the restriction of the holonomy representation of $\mu_{n,i}$ to the subsurface $S_{i,j}$ and let $\rho_{\infty, i,j}$ denote the holonomy of $\mu_{\infty,i,j}$. By Remark \ref{rmk:degeneration}, we know that $\rho_{n,i,j}$ converges to $\rho_{\infty, i,j}$, therefore, for every curve $\gamma \in \PP\setminus \D'$, the $\B$-lengths and the $\B$-twist-bend parameters for $\mu_{n}$ converge to those for $\mu_{\infty}$, and the same holds the $\B$-lengths parameters of curves in $\gamma \in \D'$. For every curve $\gamma \in \D''$, we have the parameters $\delta_{\mu_{n}}(\gamma)$ and $\delta_{\mu_{\infty}}(\gamma)$ for every $n \in \N$. In the proof of Theorem \ref{thm:coord}, we showed that these converge if and only if the boundary at infinity of the developing maps converge, which is true in our case from Remark \ref{rmk:degeneration}. Les us now consider the behaviour of the $\B$-twist-bend parameters associated to curves $\gamma \in \D'\setminus \D''$. This is the most interesting case, as these are defined for every $\mu_{n}$, but are not defined for $\mu_{\infty}$. We distinguish two cases. \\
{\emph{Case 1.}} If $\hol_{\mu_{\infty}}(\gamma)$ is semi-loxodromic or parabolic, then $\delta_{\mu_{\infty}}(\gamma)=0$. Moreover, by the convergence of the $\B$-lengths, we have that $|\ell_{\B}(\hol_{\mu_{n}}(\gamma))|$ tends to $0$. Therefore, 
    \begin{align*}
        \lim_{n \to \infty}\mathcal{H}(\mu_{n})&=\lim_{n \to \infty}H(\ell_{\B}(\hol_{\mu_{n}}(\gamma)), e^{i\Ree(\theta_{\mu_{n}}(\gamma))}, \Imm(\theta_{\mu_{n}}(\gamma)))\\
        &=(\Ima(\ell_{\B}(\hol_{\mu_{\infty}}(\gamma))),0,0,0)=\mathcal{H}(\mu_{\infty}(\gamma)) \ .
    \end{align*} 
{\emph{Case 2.}} If $\hol_{\mu_{\infty}}(\gamma)$ is loxodromic, then $\delta_{\mu_{\infty}}(\gamma)\neq 0$, and let us assume that it is positive (a similar argument applies to the negative case), which means that the sawtooth in the boundary at infinity of the developing map is future-directed. Notice that 
    \begin{align*}
        \mathcal{H}(\mu_{\infty}(\gamma))&=\big(\Imm(\ell_{\B}(\hol_{\mu_{\infty}}(\gamma)),0,0, |\ell_{\B}(\hol_{\mu_{\infty}}(\gamma))|)\big)\\
        &=\lim_{n \to +\infty}\mathcal{H}(\mu_{n}(\gamma))
    \end{align*}
    if $\Imm(\theta_{\mu_{n}}(\gamma))$ tends to $+\infty$. The fact that $\mu_{n} \in \GH(S, \D')$ but $\mu_{\infty}\in \GH(S,\D'')$ implies that the imaginary part of the re-normalised $\B$-twist-bend parameters must be unbounded, because otherwise, we could find a sub-sequence of $\mu_{n}$ converging to a point in $\GH(S,\D''\setminus \{\gamma\})$ up to the action of $G_{\D}$. Therefore, we only need to understand how the divergence of the imaginary part of the re-normalised $\B$-twist-bend parameters is related to the sawtooth appearing in the boundary at infinity of the developing map. Let $i$ and $j$ be indices such that $\gamma$ bounds $S_{i,j}$. Recall that, by Remark \ref{rmk:degeneration}, the convergence of $\mu_{n}$ implies the convergence of the boundary at infinity of the developing maps. \\
{\bf{Claim 4.5}} If $\Imm(\theta_{\B}(\hol_{\mu_{n}}(\gamma)))$ diverges to $+\infty$ (resp. $-\infty$), then $\partial_{\infty}(\dev_{\mu_{n}})$ develops a future-directed (resp. past-directed) sawtooth in the limit joining the endpoints of $\hol_{\mu_{\infty}}(\gamma)$. \\
Assuming the claim for now, we see that from our assumptions $\Imm(\tw_{\B}(\hol_{\mu_{n}}(\gamma)))$ must indeed tend to $+\infty$, because we already excluded the existence of sub-sequences with a finite limit and we can also exclude the existence of a sub-sequence $\mu_{n_{k}}$ such that $\Imm(\tw_{\B}(\hol_{\mu_{n_{k}}}(\gamma)))$ diverges to $-\infty$, because this would imply that the boundary curves $\partial_{\infty}(\dev_{\mu_{n_{k}}})$ develop a past-directed sawtooth which contradicts our assumptions on $\partial_{\infty}(\dev_{\mu_{\infty}})$. 
Combining all, this shows that $\Phi_{\PP, \D}(\mu_{n})$ converges to $\Phi_{\PP,\D}(\mu_{\infty})$. \\
\indent We are left to prove that the map $\Phi_{\PP,\D}$ is proper. Let $\mu_{n}, \mu \in \GH(S)^{aug,\D}$ be such that $p_{n}=\Phi_{\PP,\D}(\mu_{n})$ converges to $p=\Phi_{\PP, \D}(\mu)$. We have to show that $\mu_{n}$ converges to $\mu$. Let $\D'\subset \D$ be such that $\mu\in \GH(S,\D')$. Writing $p$ as in Equation (\ref{eq:longcoord}), we see that $p\in \Phi_{\PP,\D}(\GH(S,\D'))$ if and only if for every curve $\gamma_{j}\in \D'$ the corresponding coordinates $b_{j}$ and $c_{j}$ vanish. Therefore, we can assume that there is a $\D''\subset \D' \subset \D$ such that $\mu_{n}\in \GH(S,\D'')$ for all $n$ and $\mu \in \GH(S,\D')$. By Lemma \ref{lm:degeneration2}, it is then sufficient to show convergence of the holonomy and of the boundary at infinity of developing maps restricted to each connected component of $S\setminus \D'$. But this follows immediately from the convergence of the coordinates of $p_{n}$ to the coordinates of $p$ by the remark we already made in the proof of the surjectivity that the holonomy and the developing map of $\Phi_{\PP,\D}^{-1}(p)$ depend continuously on the coordinates of $p$.
\end{proof}

\begin{proof}[Proof of Claim 4.5] Exploiting the action of $G_{\D}$, for every $n \in \N$ we can assume that the $\B$-twist bend parameter $\tw_{\B}(\hol_{\mu_{n}}(\gamma))$ has real part in the interval between $0$ and $\Ree(\ell_{\B}(\hol_{\mu_{n}}(\gamma)))$ without changing $\Imm(\theta_{\B}(\hol_{\mu_{n}}(\gamma)))$. In particular, we can assume, since the $\B$-lengths converge, that the real part of $\tw_{\B}(\hol_{\mu_{n}}(\gamma))$ is uniformly bounded in $n$. We observe that
\[
    \Ree(\theta_{\B}(\hol_{\mu_{n}}(\gamma)))=\frac{2\pi\Ree(\tw_{\B}(\hol_{\mu_{n}}(\gamma)))-\Imm(\theta_{\B}(\hol_{\mu_{n}}(\gamma)))\Imm(\ell_{\B}(\hol_{\mu_{n}}(\gamma)))}{\Ree(\ell_{\B}(\hol_{\mu_{n}}(\gamma)))}
\]
so that
\begin{align*}
    \lim_{n \to +\infty}&\frac{\Ree(\theta_{\B}(\hol_{\mu_{n}}(\gamma)))\Imm(\ell_{\B}(\hol_{\mu_{n}}(\gamma)))}{\Ree(\ell_{\B}(\hol_{\mu_{n}}(\gamma)))\Imm(\theta_{\B}(\hol_{\mu_{n}}(\gamma)))}\\
        &=\frac{2\pi\Imm(\ell_{\B}(\hol_{\mu_{n}}(\gamma)))\Ree(\tw_{\B}(\hol_{\mu_{n}}(\gamma)))-\Imm(\theta_{\B}(\hol_{\mu_{n}}(\gamma)))\Imm(\ell_{\B}(\hol_{\mu_{n}}(\gamma)))^{2}}{\Ree(\ell_{\B}(\hol_{\mu_{n}}(\gamma)))^{2}\Imm(\theta_{\B}(\hol_{\mu_{n}}(\gamma)))}\\
        &=-\frac{\Imm(\ell_{\B}(\hol_{\mu_{\infty}}(\gamma)))^{2}}{\Ree(\ell_{\B}(\hol_{\mu_{\infty}}(\gamma)))^{2}}
\end{align*}
because the $\B$-lengths are uniformly bounded. 
In other words, as $n$ goes to infinity,
\begin{align*}
    \Ree(\theta_{\B}(\hol_{\mu_{n}}(\gamma)))&\Imm(\ell_{\B}(\hol_{\mu_{n}}(\gamma)))\\
    &=-\frac{\Imm(\ell_{\B}(\hol_{\mu_{\infty}}(\gamma)))^{2}}{\Ree(\ell_{\B}(\hol_{\mu_{\infty}}(\gamma)))^{2}}\Ree(\ell_{\B}(\hol_{\mu_{n}}(\gamma)))\Imm(\theta_{\B}(\hol_{\mu_{n}}(\gamma)))+o(1)
\end{align*}
from which we deduce that
\begin{align*}
    2\pi&\Imm(\tw_{\B}(\hol_{\mu_{n}}(\gamma)))\\
    &=-\frac{\Imm(\ell_{\B}(\hol_{\mu_{\infty}}(\gamma)))^{2}}{\Ree(\ell_{\B}(\hol_{\mu_{\infty}}(\gamma)))^{2}}\Ree(\ell_{\B}(\hol_{\mu_{n}}(\gamma)))\Imm(\theta_{\B}(\hol_{\mu_{n}}(\gamma)))\\
    & \ \ \ +\Ree(\theta(\hol_{\mu_{n}}(\gamma)))\Imm(\ell_{\B}(\hol_{\mu_{n}}(\gamma)))+o(1)\\
    &=\frac{\Ree(\ell_{\B}(\hol_{\mu_{n}}(\gamma)))|\ell_{\B}(\hol_{\mu_{\infty}}(\gamma))|^{2}}{\Ree(\ell_{\B}(\hol_{\mu_{\infty}}(\gamma)))^{2}}\Imm(\theta_{\B}(\hol_{\mu_{n}}(\gamma)))+o(1) \xrightarrow{n \to \infty} +\infty
\end{align*}
as the $\hol_{\mu_{\infty}}(\gamma)$ is loxodromic by assumption. Therefore, if we choose real numbers $\lambda_{n}$ and $\mu_{n}$ such that
\[
    \tw(\hol_{\mu_{n}}(\gamma))=\frac{\lambda_{n}+\mu_{n}}{2}+\tau\frac{\lambda_{n}-\mu_{n}}{2} \ ,
\]
we must necessarily have $\lim_{n\to +\infty}\lambda_{n}=+\infty$ and $\lim_{n \to +\infty}\mu_{n}=-\infty$, because the real part is uniformly bounded and the imaginary part goes to $+\infty$. \\
Let us now assume that the curve $\gamma$ bounds two different pairs of pants $P_{1}$ and $P_{2}$, where $P_{1}$ is on the left with respect to $\gamma$. The computation for the other case where $\gamma$ is contained in only one pair of pants is similar and left to the reader. Let $\alpha$ be the boundary of $P_{1}$ that follows $\gamma$ in counter-clockwise order and let $\beta$ be the boundary curve of $P_{2}$ that follows $\gamma$ in clockwise order, both oriented so that the pair of pants lies on their left. Up to conjugation, we can assume that for every $n \in \N$, 
\[
    \hol_{\mu_{n}}(\gamma)^{-}=[1,0], \ \ \ \hol_{\mu_{n}}(\gamma)^{+}=[0,1] \ \ \ \text{and} \ \ \ \hol_{\mu_{n}}(\alpha)^{-}=[-1,1] \ .
\]
The attracting fixed point of $\hol_{\mu_{n}}(\beta)$ has then coordinates (see Section \ref{sec:coord})
\[
    \hol_{\mu_{n}}(\beta)^{+}=[e^{\lambda_{n}}e^{+}+e^{\mu_{n}}e^{-},1] \xrightarrow{n \to \infty} [e^{+}, e^{-}]
\]
which is the third vertex in the future-directed sawtooth joining $[1,0]$ and $[0,1]$. Using the identification $\Pp\B^{1} \cong \R\Pp^{1}\times \R\Pp^{1}$, the curves $\partial_{\infty}(\dev_{\mu_{n}})$ are graphs of $1$-Lipschitz maps (\cite{barbot2008}) that have been re-normalised so to send one fixed point to one fixed point. Hence they converge to a limit curve $\Gamma_{\infty}$, which is itself the graph of a $1$-Lipschitz function. Since $\hol_{\mu_{n}}(\beta)^{+} \in \partial_{\infty}(\dev_{\mu_{n}})$ for every $n$, we have that $[e^{+},e^{-}]$ belongs to $\Gamma_{\infty}$. Now, graphs of $1$-Lipschitz functions are achronal curves in the boundary at infinity of anti-de Sitter space with the property that if two points lie on a light-like segment, then this segment must be entirely contained in the curve (\cite[Lemma 1.1]{Tambu_poly}). This implies that $\Gamma_{\infty}$ contains a future-directed sawtooth joining $\hol_{\mu_{\infty}}(\gamma)^{-}$ and $\hol_{\mu_{\infty}}(\gamma)^{+}$.
\end{proof}

\bibliographystyle{alpha}
\bibliography{bs-bibliography}

\bigskip

\noindent \footnotesize \textsc{DEPARTMENT OF MATHEMATICS, RICE UNIVERSITY}\\
\emph{E-mail address:}  \verb|andrea_tamburelli@libero.it|

\end{document}